\documentclass[oneside,12pt]{amsart}
\usepackage{eurosym}
\usepackage{amsfonts}
\usepackage{amsmath,amssymb,amsthm,color}
\usepackage{amsopn}
\usepackage{latexsym}
\usepackage{float}
\usepackage{bm}

\newcommand{\mg}{\mathfrak g }

\newcommand{\mn}{\mathfrak n }

\newcommand{\mz}{\mathfrak z }

\newcommand{\mv}{\mathfrak v }
\newcommand{\mh}{\mathfrak h }
\newcommand{\ma}{\mathfrak a }

\newcommand{\so}{\mathfrak{so} }

\newcommand{\bil}{g}
\newcommand{\lela}{ g(}
\newcommand{\rira}{)}
\newcommand{\lra}{\longrightarrow}

\newcommand{\bs}{\backslash}

\renewcommand{\Im}{\rm Im\,}
\newcommand{\R}{\mathbb R}
  
\newcommand{\Q}{\mathbb Q}
\newcommand{\N}{\mathbb N}

\DeclareMathOperator{\End}{End}

\DeclareMathOperator{\ad}{ad}

\numberwithin{equation}{section}

 \newtheorem{teo}{Theorem}[section]
 \newtheorem{pro}[teo]{Proposition}
 \newtheorem{cor}[teo]{Corollary}
 \newtheorem{lm}[teo]{Lemma}
 \newtheorem{defi}[teo]{Definition}

 \theoremstyle{definition}
 \newtheorem{ex}[teo]{Example}
 \newtheorem{remark}[teo]{Remark}

\newcommand{\nc}{\newcommand}
\nc{\Iso}{\operatorname{Iso}}
 \nc{\iso}{\mathfrak{iso}}
 \nc{\sso}{\mathfrak{so}}
\nc{\Ad}{\operatorname{Ad}} 
\nc{\Sym}{\mathrm{Sym}}
 \nc{\pr}{\operatorname{pr}} 
 \nc{\Dera}{\operatorname{Dera}} \nc{\Auto}{\operatorname{Auto}}
 \nc{\name}{decomposable }
 \nc{\nameend}{decomposable}

\begin{document}

\title{Symmetric Killing tensors on nilmanifolds}
\author{Viviana del Barco}
\address{Universidade Estadual de Campinas, Campinas, Brazil and Universidad Nacional de Rosrio, CONICET, Rosario, Argentina}
\thanks{V. del Barco supported by FAPESP
grants 2015/23896-5 and 2017/13725-4. }
\email{delbarc@ime.unicamp.br}

\author{Andrei Moroianu}
\address{Laboratoire de Math\'ematiques d'Orsay, Univ. Paris-Sud, CNRS, Universit\'e Paris-Saclay, 91405 Orsay, France }
\email{andrei.moroianu@math.cnrs.fr}

\begin{abstract} We study left-invariant symmetric Killing 2-tensors on $2$-step nilpotent Lie groups endowed with a left-invariant Riemannian metric, and construct genuine examples, which are not linear combinations of parallel tensors and symmetric products of Killing vector fields. 
\end{abstract}

\maketitle

\section{Introduction}

A symmetric Killing tensor on a Riemannian manifold $(M,g)$ is a smooth function defined on the tangent bundle of $M$ which is polynomial in the vectorial coordinates and constant along the geodesic flow.
As such, symmetric Killing tensors have been extensively studied by Physicists in the context of integrable systems since they define first integrals of the equation of motion. More recently, they started to be subject of systematic research in differential geometry. An account of their properties and constructions in a modern language can be found in \cite{He18,HMS16,Se03}. In this context, symmetric Killing tensors are characterized as symmetric tensor fields having vanishing symmetrized covariant derivative.

Trivial examples of symmetric Killing tensors on Riemannian manifolds are the parallel ones, as well as the Killing vector fields. Symmetric tensor products of these particular tensors generate the subalgebra of {\em decomposable} symmetric Killing tensors. Thus the following question arises: when does the this subalgebra coincide with the whole algebra of symmetric Killing tensors? This is an open question in general (see for instance, \cite[Question 3.9]{BMMT}), but it has been already considered in some particular contexts. 
For instance, in the case of Riemannian manifolds with constant curvature, it is known that every symmetric Killing tensor is decomposable \cite{Ta83,Th86}. On the contrary, there are Riemannian metrics on the 2-torus possessing {\em indecomposable} symmetric Killing 2-tensors \cite{HMS17}, \cite{MS}. One should notice that this case is extreme because these metrics lack of Killing vector fields.

In this paper we consider this question in the context of 2-step nilpotent Lie groups endowed with a left-invariant Riemannian metric. Our challenge is to establish conditions for these Riemannian Lie groups to carry indecomposable left-invariant symmetric Killing 2-tensors, that is, which are not linear combinations of parallel tensors and symmetric products of Killing vector fields.

First integrals of the geodesic flow given by a left-invariant metric on 2-step nilpotent Lie groups have been studied by several authors \cite{Bu03, KOR,MSS, Ov18}, with particular focus on the integrability of the dynamical system.  These works illustrate the difficulties and the importance of constructing first integrals of the equation of motion. Notice that in our terminology, indecomposable symmetric Killing tensor correspond to first integrals which are independent from the trivial ones (defined by the Killing vector fields and the parallel tensors).

Our starting point is the characterization of left-invariant symmetric 2-tensors on Riemannian 2-step nilpotent Lie groups which satisfy the Killing condition (see Proposition \ref{pro.caracterizKilling}). Afterwards, we address the decomposability of these symmetric Killing 2-tensors.  

Along our presentation we will stress on the fact that decomposable left-invariant symmetric Killing tensors can be obtained as a linear combination of symmetric products of Killing vector fields which {\em are not} left-invariant. 
This means that the problem can not be reduced to solve linear equations in the Lie algebra. 
We overcome this difficulty by combining  the aforementioned characterization of symmetric Killing tensors together with the description of the Lie algebra of Killing vector fields of a nilpotent Lie group given by Wolf \cite{Wo63}. We obtain in Proposition \ref{pro.Senv} and Theorem \ref{pro.necT} that the decomposability of a symmetric Killing 2-tensor depends on the possibility to extend certain linear maps to skew-symmetric derivations of $\mn$. 

The algebraic structure of 2-step nilpotent Lie algebras of dimension $\leq 7$ allows the existence of a large amount of skew-symmetric derivations. In particular, it turns out that the linear maps involved in Theorem \ref{pro.necT} always extend to skew-symmetric derivations. As a consequence, we obtain that every left-invariant symmetric Killing 2-tensor in Riemannian 2-step nilpotent Lie groups of dimension $\leq 7$ is decomposable. On the contrary, we show that there are symmetric Killing tensors on 2-step nilpotent Lie algebras of dimension 8 which are indecomposable. We actually provide a general construction method of indecomposable symmetric Killing tensors on 2-step nilpotent Lie algebras in any dimension $\geq 8$.

Taking compact quotients of nilpotent Lie group admitting indecomposable symmetric Killing 2-tensors, we obtain new examples of compact Riemannian manifolds with indecomposable symmetric Killing tensors but also admitting non-trivial Killing vector fields. Other examples were constructed in \cite{MS,MF}.

A brief account on the organization of the paper is the  following. 
In Section 2 we introduce notations and give the preliminaries on nilpotent Lie groups endowed with a left-invariant Riemannian metric. We include the description of the Lie algebra of Killing vector fields as consequence of the result of Wolf on the isometry groups of nilpotent Lie groups endowed with a left-invariant Riemannian metric. Section 3 reviews the basic results on symmetric Killing tensors. In Section 4 we characterize the symmetric endomorphisms on $\mn$ giving rise to left-invariant Killing 2-tensors on the Lie group $N$. We then show that the space of symmetric Killing 2-tensors decomposes according to the natural decomposition of a 2-step nilpotent Lie algebra. 
Section 5 contains our main results on the decomposability of left-invariant symmetric Killing 2-tensors. 
We exhibit families of Lie algebras all of whose left-invariant symmetric Killing 2-tensors are decomposable, i.e. linear combinations of parallel symmetric tensors and symmetric products of Killing vector fields. On the other hand, we show that this is not always the case by constructing examples of 2-step nilpotent Lie algebras admitting indecomposable left-invariant symmetric Killing 2-tensors. 

We finally obtain in the appendix several results about 2-step nilpotent Lie groups. 
These include relevant results on parallel symmetric tensors which are directly used in Section 5 of the paper, but we also present some more general results which, we believe, are of independent interest.

\noindent {\bf Acknowledgment.} The research presented in this paper was done during the visit of the first named author at the Laboratoire de Math\'ematiques d'Orsay, Univ. Paris-Sud. She expresses her gratitude to the LMO for its hospitality and to Funda\c{c}\~ao de Amparo \`a Pesquisa do Estado de S\~ao Paulo (Brazil) for the financial support.

Both authors are grateful to the anonymous referee for a careful reading of the manuscript and useful suggestions which helped to improve the presentation.

\section{Preliminaries on nilpotent Lie groups}\label{sec.preliminaries}

This section intends to fix notations and to summarize the most relevant features of the geometry of nilpotent Lie groups endowed with a left-invariant metric.

Let $N$ be a Lie group and denote the left and right translations by elements of the group, respectively, by
$$ L_h:N\lra N, \; L_h(p)=hp,\qquad R_h:N\lra N, \;R_h(p)=ph,\quad h,p\in N.$$ Denote by $I_h$ the conjugation by $h\in N$ on $N$, that is, $I_h=L_h\circ R_{h^{-1}}$. As usual $\Ad(h)$ denotes the differential of $I_h$ at $e$.

For each $h\in N$, $L_h$ is a diffeomorphism of $N$ so its differential at $p\in N$, $(L_h)_{*p}:T_pN\lra T_{hp}N$ is an isomorphism. This map extends to an isomorphism from the space of $(k,l)$ tensors $T_pN^{\otimes^k}\otimes T_p^*N^{ \otimes^l}$ to $T_{hp}N^{\otimes^k}\otimes T_{hp}^*N^{ \otimes^l}$, which we also denote $(L_{h})_{*p}$; we will sometimes omit the point $p$ and simply write $(L_{h})_*$. Given a differential form $\omega\in \Gamma(\Lambda^*N)$, the pullback $L_h^*\omega$ is the differential form given by $(L_h)_*\omega=(L_h^*)^{-1}\omega$.

From now on we assume that $N$ is a connected and simply connected nilpotent Lie group and we denote $\mn$ its Lie algebra. Then the Lie group exponential mapping $\exp:\mn\lra N$ is a diffeomorphism \cite{VAR}. We shall use this fact to write tensors fields on $N$ as functions defined on $\mn$ taking values on its tensor algebra of $\mn$. 

Let $S$ be a tensor field on $N$, that is, a section of the bundle $\bigoplus_{k,l}TN^{\otimes^k}\otimes T^*N^{ \otimes^l}$. For each $p\in N$, $S_p\in \bigoplus_{k,l}T_pN^{\otimes^k}\otimes T_p^*N^{ \otimes^l}$ so that $S$ defines a function $\Omega_S:\mn\lra \bigoplus_{k,l}\mn^{\otimes^k}\otimes \mn^{ *\otimes^l}$ as follows:
\begin{equation}\label{eq.defOmegaS}
\Omega_S(w):=(L_{\exp (-w)})_{*} S_{\exp w},\quad w\in \mn.
\end{equation} The map $\Omega_S$ gives the value of $S$ at $\exp w$, translated to $\mn$ by the appropriate left-translation. This is a well defined map since $\exp:\mn\lra N$ is a diffeomorphism. Notice that $S_p=(L_p)_{*e}\Omega_S(\exp ^{-1}p)$, $p\in N$.

With this notation, we can introduce the usual concept of left-invariant tensors as follows:
\begin{defi}
A left-invariant tensor field $S$ on $N$ is a tensor field such that  $\Omega_S$ is  a constant function.
\end{defi}

A vector field $X$ on $N$ with $X_e=x$ is left-invariant if and only if $\Omega_X(w)=x$ for all $w\in \mn$, that is, 
$ (L_{\exp-w})_{*}X_{\exp w}=x$. Replacing $p=\exp w$, this condition reads as the usual condition $X_p=(L_p)_{*e}x$ for all $p\in N$.

A Riemannian metric $\bil$ on $N$ is a section of the bundle $T^*N^{\otimes^2}$. Evaluated at the identity $e\in N$, $\bil_e$ defines an inner product on $\mn$ and $\Omega_{\bil}(0)=\bil_e$. The Riemannian metric $\bil$ is left-invariant if and only if $\Omega_{\bil}(w)=\bil_e$ for all $w\in \mn$. From \eqref{eq.defOmegaS}, $\bil$ is left-invariant if and only if $(L_p)_{*}$ is an isometry for all $p\in N$, or $(L_p)^*\bil=\bil$.

\begin{remark}\label{rem:leftS}
There is  a one to one correspondence between left-invariant tensor fields on $N$, and elements in the tensor algebra of $\mn$. We will sometimes identify the corresponding objects.
For instance, a left-invariant symmetric $2$-tensor on $N$ will be identified with an element in $\Sym^2\mn$ and vice-versa.
\end{remark}

Fix $\bil$ a left-invariant Riemannian metric on $N$. We denote also by $g$ the inner product it defines in $\mn$. Let $\nabla$ denote the Levi-Civita connection associated to $\bil$. The Koszul formula evaluated on left-invariant vector fields $X,Y,Z$ reads
\begin{equation}\label{eq.Koszul}
\lela \nabla_X Y,Z\rira=\frac12\{ \lela [X,Y],Z\rira+\lela [Z,X],Y\rira+\lela [Z,Y],X\rira\}.
\end{equation}
One has then, since $(L_p)_*$ is an isometry for all $p\in N$, that $\nabla_XY$ is left-invariant whenever $X,Y$ are so. Moreover, for any left-invariant vector field $X$ and any left-invariant tensor field $S$ on $N$, the covariant derivative $\nabla_X S$ is also left-invariant. Thus $\nabla$ defines a linear map $\nabla:\mn\lra  \End(\mathcal T(\mn))$ by $X\mapsto \nabla_X$, where $\mathcal T(\mn)=\bigoplus_{k,l}\mn^{\otimes^k}\otimes \mn^{*\otimes^l}$ is the tensor algebra of $\mn$. \medskip

To continue, we describe the Killing vector fields of the Riemannian manifold $(N,\bil)$. Recall that a vector field $\xi$ on $N$ is a Killing vector field if and only if its flow is given by isometries. In terms of the covariant derivative, $\xi$ is a Killing vector field if and only if $ \lela \nabla_X \xi,X\rira=0$ for any vector field $X$.

The Riemannian manifold $(N,\bil)$ is complete because it is homogeneous. Therefore, every Killing vector field is complete and the vector space of Killing vector field is finite dimensional. Endowed with the Lie bracket of vector fields, they constitute a Lie algebra which is isomorphic to the Lie algebra $\iso(N)$ of the isometry group $\Iso(N)$ of $N$. We recall the results of Wolf on the structure of this isometry group \cite{Wo63}.

The nilpotent Lie group $N$ itself embeds as a subgroup of $\Iso(N)$ through the inclusion $g\in N\mapsto L_g\in\Iso(N)$, since the metric is left-invariant. Denote by $\Auto(N)$ the group of isometries of $N$ which are also group automorphisms; these isometries fix the identity. Wolf shows that these two subgroups determine the isometry group of $N$ and gives its precise algebraic structure.

\begin{teo}  \cite{Wo63} Let $N$ be a connected nilpotent Lie group with a left-invariant metric $\bil$. Then the isometry group $\Iso(N)$ is the semidirect product of the group of left-translations and the group of isometric automorphism, that is, 
$$\Iso(N)= {\rm Auto}(N)\ltimes N,$$
where the action of $\Auto(N)$ on $N$ is $ f\cdot L_h=L_{f(h)}$ for all $h\in N$, $f\in \Auto(N)$. 
\end{teo}

The differential of an isometric automorphism $f\in \Auto(N)$ defines an orthogonal automorphism of the Lie algebra $\mn$, $f_{*e}\in\Auto(\mn)$. Since $N$ is simply connected, this correspondence is an isomorphism and we have $\Auto(N)\simeq\Auto(\mn)$. Let $[\ ,\ ]$ denote the Lie bracket on $\mn$, then the Lie algebra of $\Auto(\mn)$ is the subspace of skew-symmetric derivations of $\mn$, that is, 
$$ \Dera(\mn)=\{D\in \so(\mn) \ |\  D[x,y]=[Dx,y]+[x,Dy]\,\;\mbox{for all }x,y\in\mn\}.$$

The result of Wolf implies that the Lie algebra of the isometry group is $\iso(N)=\Dera(\mn)\ltimes\mn$, where the semi-direct product is defined by the canonical action of $\Dera(\mn)$ on $\mn$. Therefore, we can distinguish two different types of Killing vector fields on $N$, namely, those corresponding to elements in the Lie algebra and those corresponding to skew-symmetric derivations.

The Killing vector field $\xi_x$ associated to an element $x\in \mn$ is the right-invariant vector field on $N$ whose value at $e$ is $x$: $(\xi_x)_p=(R_p)_{*e}x$ for  all $p\in N$. In fact, the flow of this vector field is given by left-translations, which are isometries.  Given $p\in N$, $(\xi_x)_p=(R_p)_{*e}x=\frac{d}{dt}|_0\exp tx\cdot p$.  But $R_p=L_p\circ I_{p^{-1}}$ hence $(\xi_x)_p=(L_p)_{*e}(\Ad(p^{-1})x)$.
If $w\in \mn$ is such that $p=\exp w$, then $\Ad(p^{-1})=\Ad(\exp(-w))$ and thus
\begin{equation}\label{eq.xix}(\xi_x)_{\exp w}=(L_{\exp w})_{*e}(\Ad(\exp(-w))).\end{equation}

The Killing vector field $\xi_D$ associated to a skew-symmetric derivation $D$ of $\mn$ is constructed as follows. The derivation gives a curve $e^{tD}$ in $\Auto(\mn)$ which at the same time induces a curve of isometries $f_t\in \Auto(N)$ by the condition $(f_t)_{*e}=e^{tD}$, so it satisfies $(\xi_D)_p=\frac{d}{dt}|_0 f_t(p)$, $p\in N$. For any automorphism $f:N\lra N$, one has $f(p)=f(\exp w)=\exp f_{*e} w$. Hence $(\xi_D)_p=\frac{d}{dt}|_0 \exp ((f_t)_{*e}w)=\frac{d}{dt}|_0 \exp (e^{tD}w)=d\exp_w(Dw)$. Taylor's formula for the differential of the exponential map (see \cite[Theorem 1.7, Ch. II]{HE}) gives 
\begin{equation}\label{eq.xiD}(\xi_D)_{\exp w}=d(L_{\exp w})_e(Dw-\frac12\ad_w Dw+\frac16\ad_w^2Dw-\ldots). \end{equation}
This sum is finite since $\mn$ is nilpotent.

Using the notation \eqref{eq.defOmegaS} of tensor fields on $N$ as functions defined on $\mn$, and the fact that $\Ad(\exp(-w))=e^{-\ad_w}$, for every $x\in \mn$ and $D\in \Dera(\mn)$ we get from \eqref{eq.xix}--\eqref{eq.xiD}:
\begin{eqnarray}
\Omega_{\xi_x}(w)&=&\Ad(\exp(-w))=x-[w,x]+\frac12[w,[w,x]]+\ldots,\quad  \forall w\in \mn.\label{eq.omegaxix}\\
\Omega_{\xi_D}(w)&=&d\exp_w(Dw)=Dw-\frac12[w,Dw]+\frac16\ad_w^2Dw-\ldots ,\quad \forall  w\in \mn.\label{eq.omegaxiD}
\end{eqnarray}

Notice that in general these Killing vector fields are not left-invariant. On the one hand $\xi_x$ is left-invariant if and only if $\Ad(\exp(-w))=x$ for all $w\in \mn$, tht is, $I_p(x)=x$ for  all $p\in N$, which occurs only when $x$ is in the center of $N$. On the other hand, $\Omega_{\xi_D}(0)=0$, so $\xi_D$ is left-invariant if and only if $D=0$.

\medskip

At this point we will focus on the geometry of $2$-step nilpotent Lie groups endowed with a left-invariant metric. The Lie algebra $\mn$ is said to be $2$-step nilpotent if it is non abelian and $\ad_x^2=0$ for all $x\in\mn$.
 
The center and the commutator of a Lie algebra $\mn$ are defined as
\[\mz=\{z\in\mn\ |\   [x,z]=0, \,\mbox{for all }x\in\mn\},\qquad
\mn'=\{ [x,y]\ |\   x,y\in\mn\}.
\]
 In the particular case of $2$-step nilpotent Lie groups, $\mn'\subset\mz$, thus \eqref{eq.omegaxix} and \eqref{eq.omegaxiD} simplify to:
\begin{equation}\label{eq.xi}
\Omega_{\xi_x}(w)=x-[w,x],\qquad\Omega_{\xi_D}(w)=Dw-\frac12[w,Dw]\quad  w\in \mn.
\end{equation}

From now on we assume that $\mn$ is $2$-step nilpotent and denote as before by $N$ the simply connected group with Lie algebra $\mn$, endowed with the left-invariant metric $g$ defined by the scalar product on $\mn$. We shall describe the main geometric properties of $(N,g)$ through linear objects in the metric Lie algebra $(\mn,g)$, following the work of Eberlein \cite{EB}.

Let $\mv$ be the orthogonal complement of $\mz$ in $\mn$ so that $\mn=\mv\oplus\mz$ as an orthogonal direct sum of vector spaces.  Each central element $z\in\mz$ defines an endomorphism $j(z):\mv\lra\mv$ by the equation
\begin{equation}\label{eq:jota}
\lela j(z)x,y\rira :=\lela z,[x,y]\rira \quad \mbox{ for all } x,y\in\mv.
\end{equation}
This endomorphism $j(z)$ belongs to $\sso(\mv)$, the Lie algebra of skew-symmetric maps of $\mv$ with respect to $\bil$. The linear  map $j: \mz \to \sso(\mv)$ captures important geometric information of the Riemannian manifold $(N,\bil)$. 
Using Koszul's formula \eqref{eq.Koszul}, the Levi-Civita covariant derivative of $g$ can be given in terms of the $j(z)$ maps by
\begin{equation}\label{eq:nabla}\left\{
\begin{array}{ll}
\nabla_x y=\frac12 \,[x,y] & \mbox{ if } x,y\in\mv,\\
\nabla_x z=\nabla_zx=-\frac12 j(z)x & \mbox{ if } x\in\mv,\,z\in\mz,\\
\nabla_z z'=0& \mbox{ if } z, z'\in\mz.
\end{array}\right.
\end{equation}

If $D:\mn\lra\mn$ is a derivation of $\mn$, then by definition $[Dz,x]=D[z,x]-[z,Dx]$ for all $z\in \mz$, $x\in \mn$, so $D$ preserves the center $\mz$. Thus a skew-symmetric derivation of $\mn$ preserves both $\mz$ and $\mv=\mz^\bot$ and for every $x,y\in\mv$ and $z\in\mz$ one has:
\begin{eqnarray*}g(j(Dz)x,y)&=&g([x,y],Dz)=-g(D[x,y],z)=-g([Dx,y],z)-g([x,Dy],z)\\&=&-g(j(z)Dx,y)-g(j(z)x,Dy)=-g([j(z),D],y),\end{eqnarray*}
whence
\begin{equation}\label{eq.deriv}
j(Dz)=[D|_\mv,j(z)] \quad \mbox{ for all } z\in \mz.
\end{equation}
Conversely, it is straightforward to check that a skew-symmetric endomorphism on $\mn$ which preserves $\mz$ and $\mv$ and satisfies \eqref{eq.deriv} is a derivation.

\section{Symmetric Killing tensors} 

Let $(V,\bil)$ be an $n$-dimensional real vector space. 
Denote by $\mathrm{Sym}^kV$ the set of symmetric tensor products in $V^{\otimes^k}$. Elements in $\mathrm{Sym}^kV$ are symmetrized tensor products 
$$v_1\cdot \ldots \cdot v_k=\sum_{\sigma\in \mathfrak{S}_k} v_1\otimes \ldots \otimes  v_k, \quad v_1, \ldots, v_k\in V,$$ 
where $\mathfrak{S}_k$ denotes the group of bijections of the set $\{1,\ldots,k\}$. Symmetric 1-tensors are simply vectors in $V$. 

The inner product $\bil$ is induced to $\Sym^kV$ by the formula
$$\lela v_1\cdot \ldots \cdot v_k, u_1\cdot \ldots \cdot u_k\rira= \sum_{\sigma\in\mathfrak{S}_k}\lela v_1, u_{\sigma(1)}\rira\cdot \ldots\cdot\lela v_k, u_{\sigma(k)}\rira, \quad u_i,v_i\in V.$$
The space of symmetric $k$-tensors on $V$ is identified with the subspace of totally symmetric covariant tensors of degree $k$. In fact, we can identify a symmetric tensor $S\in{\rm Sym}^kV$ with the multilinear map $S:V\times \cdots\times V\lra \R$ satisfying 
\begin{equation}\label{eq.identifsimtensor}
S(v_1, \ldots, v_k)=\lela S,v_1\cdot\ldots\cdot v_k\rira, \quad v_1, \ldots, v_k\in V.
\end{equation} 
Under this identification, a vector $v\in V$ (symmetric $1$-tensor) is identified with its metric dual $\lela v,\cdot\rira$. The inner product $\bil$ becomes a symmetric $2$-tensor and $\bil=\frac12 \sum_{i=1}^n e_i^2$, if $\{e_1, \ldots, e_n\}$ is an orthonormal basis  of $V$. Here we denote $e_i^2=e_i\cdot e_i$. In general,  symmetric $k$-tensor are homogeneous polynomials of degree $k$ in the variables $e_1, \ldots, e_n$. 

Symmetric $2$-tensors $S\in {\rm Sym}^2V$ can also be seen as symmetric endomorphisms on $V$. In fact, $S$ is a symmetric bilinear form  $S:V\times V\lra \R$ and is related to the inner product by the condition $S(u_1,u_2)=\lela Su_1,u_2\rira$, $u_1,u_2\in V$. Given an orthonormal basis $\{e_1, \ldots, e_n\}$, if the matrix of $S$ in this basis, as an endomorphism, is $(s_{ij})_{ij}$, then $S$ as a polynomial reads 
\begin{equation}\label{eq.S}S=\frac12 \sum_{1\leq i, j\leq n} s_{ij}\ e_i\cdot e_j.\end{equation}

Let $(M,\bil)$ be a Riemannian manifold. 
A symmetric $k$-tensor field on $M$ is a section of the bundle $\Sym^k{TM}$, which is a sub-bundle of $TM^{\otimes^k}$. Symmetric 1-tensors are the vector fields on $M$. The symmetric product of vector fields on $M$ defines symmetric 2-tensors by the formula $ (X\cdot Y)_p:=X_p\otimes Y_p+Y_p\otimes X_p$. \smallskip

Under the identification in \eqref{eq.identifsimtensor}, symmetric tensors correspond to  covariant tensor fields, for instance a vector field $X$ on $M$ corresponds to the 1-form $\lela X, \cdot\rira$. As symmetric bilinear form, $X\cdot Y$ satisfies $(X\cdot Y)(u,v)=\lela X,u\rira\lela Y,v \rira+\lela X,v\rira\lela Y,u \rira$. For every local orthonormal frame $\{E_1, \ldots, E_n\}$, the Riemannian metric $\bil$ reads $\bil=\frac12 \sum_{i=1}^n E_i\cdot E_i$. Finally, notice that symmetric $2$-tensors on $M$ correspond to sections of $\End(TM)$ which are symmetric with respect to the Riemannian metric.

Let $\nabla$ denote the Levi-Civita connection defined by $\bil$ on $M$. 
\begin{defi}\label{def.symmkilltens}
A symmetric $k$-tensor $S$ is a Killing tensor if $(\nabla_XS)(X,\ldots,X)=0$ for any vector field $X$.
\end{defi}
Equivalently, a symmetric $k$-tensor is a Killing tensor if its symmetrized covariant derivative vanishes. This notion generalizes that of Killing vector field, as we shall see below. 

A concept related to symmetric Killing tensors is that of Killing forms \cite{Se03}. A $2$-form $\omega$ is said to be a Killing form if $X\lrcorner\  \nabla_X\omega=0$ for every vector field $X$ in $M$. Given a Killing 2-form, the symmetric 2-tensor defined as $S^\omega(X,Y)=\lela (X\lrcorner \ \omega)^\#,(Y\lrcorner \ \omega)^\#\rira$  is a Killing 2-tensor (see Proposition \ref{pro.symmkilltens12}), where $(X\lrcorner\ \omega)^\#$ denotes the metric dual of the 1-form $X\lrcorner \ \omega$. The section of $\End(TM)$ corresponding to the symmetric tensor $S^\omega$ is $-T^2$.

We next give a short account of basic results on Killing tensors \cite{HMS17}.
\begin{pro}\label{pro.symmkilltens12}
\begin{enumerate}
\item A vector field $\xi$ is a symmetric Killing $1$-tensor if and only if $\xi$ is a Killing vector field.
\item The symmetric product of two Killing vector fields is a symmetric Killing $2$-tensor.
\item The Riemannian metric $\bil$ is a symmetric Killing $2$-tensor.
\item If $\omega$ is a Killing $2$-form then $S^\omega$ is a symmetric Killing $2$-tensor.
\end{enumerate}
\end{pro}

\begin{proof} The first and third statements are obvious.
To check the second one, let $\xi,\eta$ be two Killing vector fields of $M$, and let $X$ be a vector field. Then $\nabla_X(\xi\cdot \eta)=\nabla_X\xi\cdot \eta+\xi\cdot \nabla_X\eta$, which is zero. Finally, from the definition of $S^\omega$ we get $(\nabla_XS^\omega)(X,X)=2\lela (X\lrcorner \nabla_X\omega)^\#, (X\lrcorner\ \omega)^\#\rira $ for any vector field $X$. 
\end{proof}

Let now $(N,g)$ be a simply connected 2-step nilpotent Lie group endowed with a left-invariant metric.

Every symmetric $k$-tensor $S$ on $N$ defines a function $\Omega_S$ on $\mn$ as in \eqref{eq.defOmegaS}. Then $\Omega_S(w)$ is a symmetric tensor on $\mn$ for all $w\in \mn$, that is, $\Omega_S(w)\in \mathrm{Sym}^k\mn\subset\mn^{\otimes^k}$. In particular, the left-invariant metric $\bil$ satisfies $\Omega_{\bil}(w)=\frac12 \sum_{i=1}^ne_i^2$ for all $w\in\mn$, if $\{e_1, \ldots,e_n\}$ is an orthonormal basis of $\mn$. 

From the definition of symmetric product of vector fields we immediately get
\begin{equation}\label{eq.Omegaprod}
\Omega_{X_1\cdot\ldots \cdot X_k}(w)=\Omega_{X_1}(w)\cdot \ldots\cdot\Omega_{X_k}(w), \quad \mbox{ for all }w\in \mn,
\end{equation}
for every vector fields $X_1, \ldots, X_k$ on $N$. 
This allows us to compute the functions corresponding to the symmetric product the Killing vector fields given in \eqref{eq.omegaxix} and \eqref{eq.omegaxiD}: 
\begin{eqnarray}
\Omega_{\xi_x\cdot\xi_{y}}(w)&=& x\cdot y -(x\cdot[w,y]+y\cdot [w,x])+[w,x]\cdot[w,y]\label{eq.symmprodxy}\\
\Omega_{\xi_x\cdot\xi_{D}}(w)&=&x\cdot Dw -([w,x]\cdot Dw+\frac12 x\cdot [w,Dw])+\frac12 [w,x]\cdot [w,Dw]\label{eq.symmprodxD}\\
\Omega_{\xi_D\cdot\xi_{D'}}(w)&=&Dw\cdot D'w-\frac12 \left(Dw\cdot [w,D'w]+D'w\cdot [w,Dw]\right)\nonumber\\
&&\quad +\frac14 [w,Dw]\cdot [w,D'w]\nonumber
\end{eqnarray}
These functions are in general non-constant, so they define symmetric Killing $2$-tensors on $N$ which are not always left-invariant.

Recall that if $S$ is a left-invariant $k$-tensor on $N$ and $x\in \mn$ then $\nabla_x S$ is also left-invariant. Thus from Definition \ref{def.symmkilltens} we have that $S$ is a symmetric Killing tensor if and only if \begin{equation}\label{cor.symkilln}
\lela (\nabla_xS)x,x\rira =0,\quad\mbox{for all } x \in \mn.
\end{equation}

\section{Left-invariant symmetric Killing $2$-tensors on $2$-step nilpotent Lie groups}

In this section we study which symmetric endomorphisms of $\mn$ define left-invariant symmetric Killing $2$-tensors on the 2-step nilpotent Lie group $N$ endowed with a left-invariant metric $g$. We start by giving a characterization of such endomorphisms in terms of their behavior with respect to the orthogonal decomposition $\mn=\mv\oplus\mz$.

\begin{pro}\label{pro.caracterizKilling}
A left-invariant symmetric $2$-tensor on $N$, identified with an element $S\in \Sym^2\mn$  as in Remark \ref{rem:leftS}, is a Killing tensor if and only if
\begin{equation} \label{eq.killtensor}
\left\{\begin{array}{l}
[Sx,y]=[x,Sy] , \mbox{ for all $x,y\in \mv$, }\\
 \ad_x \circ S|_\mz \mbox{ is skew-symmetric on } \mz \mbox{ for all $x\in \mv$}.
 \end{array}\right.
 \end{equation}
\end{pro}
\begin{proof}
Let $S:\mn\lra\mn$ be a symmetric endomorphism. 
From \eqref{cor.symkilln}, $S$ is a Killing tensor if and only if $\lela (\nabla_yS)y,y\rira=0$ for all $y\in \mn$. By Koszul's formula we have 
\begin{equation}\label{eq.nablaxSyz}
\lela (\nabla_yS)y,y\rira=2\lela [y,Sy],y\rira,
\end{equation} and therefore $S$ is Killing if and only if 
$$\lela [y,Sy],y\rira=0\quad \mbox{ for  all }y\in\mn$$
(cf. \cite[Proposition 2.4.]{KOR}). Taking $y=x+z$ with $x\in \mv$ and $z\in \mz$, the last equality reads
\begin{equation}\label{eq.suma}
\lela z,[Sz,x]\rira+\lela z,[Sx,x]\rira=0\quad \mbox{ for all  }x\in \mv,z\in\mz.
\end{equation} The first term of this sum is quadratic in $z$, whilst the second one is linear. This implies that \eqref{eq.suma} is equivalent to
$$\lela z,[Sx,x]\rira=0\quad\mbox{ and }\quad  \lela z,[Sz,x]\rira=0\quad \mbox{ for all  }x\in \mv,z\in\mz.$$
By polarization, the first equation is equivalent to $[Sx,y]=[x,Sy]$ for all $x,y\in \mv$, and the second equation simply says that  $\ad_x\circ S|_{\mz}$ is a skew-symmetric endomorphism of $\mz$ for every $x\in \mv$.
\end{proof}

\begin{ex}\label{ex.v} Let $S\in {\Sym}^2\mn$ be a symmetric tensor such that $Sz=0$ for all $z\in \mz$ and $S|_\mv=\lambda {\rm Id}_\mv$ for some $\lambda \in\R$. It is clear that $S$ satisfies the system \eqref{eq.killtensor}, so $S$ determines a left-invariant Killing tensor on $N$.
\end{ex}

\begin{remark} Killing $2$-forms define symmetric Killing $2$-tensors as showed in Proposition \ref{pro.symmkilltens12}. On the other hand, Barberis et. al. \cite[Theorem 3.1.]{BDS}, proved that a skew-symmetric map $T:\mn\lra\mn$ defines a Killing $2$-form if and only if $T$ preserves $\mz$ and $\mv$, and verifies $[Tx,y]=[x,Ty]$ and $T[x,y]=3[Tx,y]$ for every $x,y\in \mv$ (see  also \cite{AD}).
If $T$ is such an endomorphism, then $S:=-T^2$ verifies the conditions \eqref{eq.killtensor}, so it defines indeed a symmetric Killing 2-tensor on $N$ (cf. Proposition \ref{pro.symmkilltens12} (4) above). 
\end{remark}

The decomposition $\mn=\mv\oplus\mz$ induces a splitting of the space of symmetric $2$-tensors on $\mn$ 
\begin{equation}\label{eq.decomvz}
{\rm Sym}^2\mn={\rm Sym}^2\mv\oplus (\mz \cdot \mv)\oplus {\rm Sym}^2\mz.
\end{equation} 
The elements of ${\rm Sym}^2\mv$, $\mz \cdot \mv$ and ${\rm Sym}^2\mz$ will be viewed as symmetric endomorphisms of $\mn$ by means of the standard inclusions.

\begin{defi}\label{def.m} Given $S\in {\rm Sym}^2\mn$, denote by $S^\mv$,  $S^{\rm m}$ and $S^\mz$ its components with respect to the above decomposition; we call $S^{\rm m}$ the mixed part of $S$. \end{defi}

\begin{pro} \label{pro.z} For every $S\in \Sym^2\mn$ we have
\begin{enumerate}
\item  the left-invariant symmetric tensor defined by $S^\mz$ is always Killing.
\item the left-invariant symmetric tensor defined $S$ is a symmetric Killing tensor if and only if the left-invariant symmetric tensors defined by $S^\mv$ and by $S^{\rm m}$ are Killing.
\end{enumerate}
\end{pro}
\begin{proof} (1) Any element of ${\rm Sym}^2\mz$ defines a left-invariant symmetric Killing tensor because it trivially verifies the system \eqref{eq.killtensor}.

(2) Notice that $[Sx,y]=[S^\mv x,y]$ and $ \ad_x S z=\ad_x S^{\rm m} z$ for all $x,y\in \mv$, $z\in\mz$. Moreover, $[S^{\rm m}x, y]=0$ and $S^\mv z=0$ for every $x,y\in \mv$, $z\in\mz$. Therefore, $S$ verifies the system \eqref{eq.killtensor} if and only if both $S^{\rm m}$ and $S^\mv$ verify it.
\end{proof}

\begin{defi} A $2$-step nilpotent Lie algebra $\mn$ is called non-singular if $\ad_x:\mv\lra \mz$ is surjective for all $x\in \mv\setminus \{0\}$.
\end{defi}

\begin{pro} \label{cor.Killnonsing} Let $\mn$ be a non-singular $2$-step nilpotent Lie algebra and let $S\in {\Sym}^2\mn$. If $S$ defines a Killing tensor on $N$ then $S^{\rm m}=0$.
\end{pro}

\begin{proof} Since $\mn$ is non-singular, we have ${\Im}({\ad}_y)=\mz$ for all $y\in \mv\setminus \{0\}$. Let $z$ be any element of $\mz$. Then $\lela [x,Sz],z\rira=0$ for all $x\in \mv$, because $S$ is Killing, and thus satisfies \eqref{eq.killtensor}. This implies $z\bot {\rm Im} \ad_{y}$ for $y:=\pr_\mv Sz\in \mv$, which, as $\mn$ is non-singular, implies that either $z=0$ or $y=0$. Therefore in both cases $\pr_\mv Sz=0$, meaning that $Sz\in\mz$ for every $z\in\mz$. 
\end{proof}

\begin{ex}\label{ex.heis}
Let $\mh_n$ be the Heisenberg Lie algebra of dimension $2n+1$. The Lie algebra $\mh_n$ has a basis $\{x_1,\ldots, x_n,y_1, \ldots,y_n,z\}$ where the non-trivial Lie brackets are
$[x_i,y_i]=z$ and $\mz=\R z$ is the center. This is a non-singular Lie algebra and 
$j(z)=:J$ is the canonical complex structure in $\R^{2n}$. Consider the inner product making this basis an orthonormal basis. By Proposition \ref{cor.Killnonsing}, the mixed part $S^{\rm m}$ of any symmetric Killing tensor $S$ vanishes. Moreover, the system \eqref{eq.killtensor} is equivalent to $[S^\mv,J]=0$. We conclude that $S$ is a symmetric Killing tensor on $\mh_n$ if and only if it preserves the decomposition $\mh_n=\mv\oplus \mz$ and its restriction to $\mv$ commutes with $J$ (cf. \cite[Theorem 3.2]{KOR}).
\end{ex}

The following example shows that there exist $2$-step nilpotent metric Lie algebras carrying left-invariant symmetric Killing $2$-tensors $S$ with $S^{\rm m}\ne 0$, that is, not preserving the decomposition $\mv\oplus\mz$. 

\begin{ex}\label{ex.2step}
Let $\mn$ be the Lie algebra of dimension $6$ having an orthonormal basis $\{e_1, \ldots, e_6\}$ satisfying the non-zero bracket relations $$[e_1,e_2]=e_4, \quad [e_1,e_3]=e_5,\quad[e_2,e_3]=e_6.$$
Then $\mz= {\rm span}\{e_4, e_5,e_6\}=\mn'$ and $\mv={\rm span}\{e_1,e_2,e_3\}$. 

An element $S\in \mz\cdot \mv$ defines a symmetric Killing tensor on the simply connected Lie group with Lie algebra $\mn$ if and only if $\lela \ad_x S z,z\rira=0$ for all $x\in\mv$, $z\in \mz$. Let $Se_4=ae_1+be_2+ce_3$, then $b=\lela [e_1,S e_4] ,e_4\rira$ and $a=\lela [e_2,S e_4] ,e_4\rira$ are zero thus $Se_4=ce_3$. Analogously, one obtains that $Se_5$ and $Se_6$ are multiples of $e_2$ and $e_1$, respectively. Moreover, $\lela[e_2,Se_6],e_4\rira+\lela[e_2,Se_4],e_6\rira=0$ implies that $S(e_6)=-ce_1$, and $\lela[e_1,Se_5],e_4\rira+\lela[e_1,Se_4],e_5\rira=0$ implies that $S(e_5)=ce_2$, so finally $S=c(e_1\cdot e_6-e_2\cdot e_5+e_3\cdot e_4)$.

Proposition \ref{pro.z} (1) and Example \ref{ex.v} show that every symmetric $2$-tensor in $\Sym^2\mz$ and the symmetric tensor $\sum_{i=1}^3 e_i\cdot e_i\in \Sym^2\mv$ define left-invariant Killing tensors on $N$. Therefore, the same holds for
$$a_1 \sum_{i=1}^3 e_i\cdot e_i+a_2(e_1\cdot e_6-e_2\cdot e_5+e_3\cdot e_4)+\sum_{4\leq i\leq j\leq 6}a_{ij}e_i\cdot e_j, \qquad a_i,a_{ij}\in \R.$$
It will follow from the results in the next section that any left-invariant symmetric Killing $2$-tensor on $N$ is induced by an element of this form.
\end{ex}
\medskip

To continue we study left-invariant symmetric Killing $2$-tensors on Riemannian products of 2-step nilpotent Lie groups. A nilpotent Lie algebra endowed with an inner product $(\mn,\bil)$ is called {\em reducible} if it can be written as an orthogonal direct sum of ideals $\mn=\mn_1\oplus\mn_2$. In this case we endow $\mn_i$ with the inner product $\bil_i$ which is the restriction of $\bil$ to $\mn_i$, for each $i=1,2$. Otherwise, $(\mn,g)$ is called {\em irreducible}.

\begin{remark} \label{rem.j}
\begin{enumerate}
\item If $\mn$ is a reducible 2-step nilpotent Lie algebra, then $\dim\mz\geq 2$. Indeed, if it decomposes as a direct sum of orthogonal ideals $\mn=\mn_1\oplus\mn_2$, then the center of $\mn$ is $\mz=\mz_1\oplus\mz_2$ where $\mz_i$ is the center of $\mn_i$, and both $\mz_i$ are non-zero.
\item \label{it.remj} If $\mn$ is an irreducible 2-step nilpotent Lie algebra, then $j:\mz\lra \so(\mv)$ is  injective.
 Indeed, it is easy to check that the kernel  $\ma$ of $j:\mz\lra \so(\mv)$ (which is equal to $\mz\cap \mn'^\bot$) and its  orthogonal $\ma^\bot$ in $\mn$ are both ideals of $\mn$. Therefore, if $\mn$ is irreducible, then $\ma=0$ and thus  $j$ is injective.
 \end{enumerate}
 \end{remark}

The next results aim to show that left-invariant symmetric Killing $2$-tensors on direct products of 2-step nilpotent Lie groups are determined by left-invariant symmetric Killing $2$-tensors on its factors.

Let $\mn$ be a 2-step nilpotent Lie algebra which is decomposable and let $\mn_1$, $\mn_2$ be orthogonal ideals such that $\mn=\mn_1\oplus\mn_2$. Then
\begin{eqnarray}
\label{eq.n12}
\Sym^2\mn=\Sym^2\mn_1\oplus (\mn_1\cdot \mn_2)\oplus \Sym^2\mn_2.
\end{eqnarray}
Given $S\in \Sym^2\mn$ we denote $S_i\in \Sym^2\mn_i$ and $S_{\rm nd}\in
\mn_1\cdot\mn_2$ the symmetric tensors in $\mn$ such that $S=S_1+S_{\rm nd}+S_2$
($S_{\rm nd}$ is the non-diagonal part of $S$).

\begin{pro} \label{pro.idealsn}
Let $S\in \Sym^2\mn$.
\begin{enumerate}
\item The left-invariant tensor defined by $S$ is a Killing tensor if and only if
$S_1, S_2,S_{\rm nd}$ are Killing tensors on $\mn$.
\item If $S$ defines a left-invariant Killing tensor then its non-diagonal part $S_{\rm nd}$ satisfies
$S_{\rm nd}\in \mz_1\cdot \mz_2$.
\end{enumerate}
\end{pro}

\begin{proof} (1) If $S$ is Killing, for every $x=x_1+x_2$ and $y=y_1+y_2$ in
$\mv=\mv_1\oplus\mv_2$ and  for every $z=z_1+z_2\in\mz=\mz_1\oplus\mz_2$, we have by
\eqref{eq.killtensor}:
$$[S_1x,y]-[x,S_1y]=[S_1x_1,y_1]-[x_1,S_1y_1]=[Sx_1,y_1]-[x_1,Sy_1]=0$$
and 
$$g(z,[S_1z,x])=g(z_1,[S_1z_1,x_1])=g(z_1,[Sz_1,x_1])=0,$$
thus showing that $S_1$ is Killing, and similarly $S_2$ is Killing too. By
linearity, $S_{\rm nd}$ is also Killing. The converse is clear.

(2) If $S$ is Killing, then $S_{\rm nd}$ is Killing so for all
$z=z_1+z_2\in\mz=\mz_1\oplus\mz_2$ and $x_1\in\mv_1$ we have
$$0=g(z,[S_{\rm nd}z,x_1])=g(z_1,[S_{\rm nd}z_2,x_1]),$$
thus showing that the subspace $S_{\rm nd}(\mz_2)$ of $\mn_1$ is actually contained
in $\mz_1$. 
\end{proof}

It will be useful later to compare the above decomposition of $S$ on a direct sum $\mn_1\oplus \mn_2$ with the one introduced in Definition \ref{def.m}. Consider the decomposition $\mn=\mv\oplus\mz$ and $\mn_i=\mv_i\oplus\mz_i$ for $i=1,2$ (one of
the $\mv_i$ might also vanish), then $\mz=\mz_1\oplus\mz_2$ and $\mv=\mv_1\oplus\mv_2$.

\begin{pro} \label{pro.decred} If $S\in \Sym^2(\mn_1\oplus \mn_2)$ defines a left-invariant Killing tensor on $N$ then $S^\mv=S_1^{\mv_1}+S_2^{\mv_2}$, $S^\mz=S_1^{\mz_1}+S_2^{\mz_2}+S_{\rm nd}$, and its mixed part is $S^{\rm m}=S_1^{\rm m}+S_2^{\rm m}$, where $S_i^{\rm m}$ denotes the mixed part of $S_i$ as a tensor on $\mn_i$.
\end{pro}
\begin{proof}
Consider the decompositions $S_i=S_i^{\mv_i}+S_i^{\mz_i}+S_i^{\rm m}$ of $S_i$ from Definition \ref{def.m}. In the expression 
$$S=(S_1^{\mv_1}+S_2^{\mv_2})+(S_1^{\rm m}+S_2^{\rm m})+(S_1^{\mz_1}+S_2^{\mz_2}+S_{\rm nd})$$
we clearly have $S_1^{\mv_1}+S_2^{\mv_2}\in \mathrm{Sym}^2\mv$, $S_1^{\rm m}+S_2^{\rm m}\in\mv\cdot\mz$, and
according to Proposition \ref{pro.idealsn}, $S_1^{\mz_1}+S_2^{\mz_2}+S_{\rm nd}\in \mathrm{Sym}^2\mz$. 
\end{proof}

\section{(In)decomposable symmetric Killing tensors}

Let $N$ be a 2-step nilpotent Lie group endowed with a left-invariant metric $\bil$ and let $\mn$ be its Lie algebra. By Proposition \ref{pro.symmkilltens12}, every linear combination of symmetric products of Killing vector fields is a symmetric Killing 2-tensor on $N$. The metric and, more generally, any parallel symmetric Killing tensors are trivially Killing tensors. We are interested in Killing tensors which are not of these two types, so we introduce the following terminology.

\begin{defi} A left-invariant symmetric Killing $2$-tensor $S$ is called {\em \name} if it is the sum of a parallel tensor and a linear combination of symmetric products of Killing vector fields. If this is not the case, we say that $S$ is {\em indecomposable}. \end{defi}

Note that from Theorem \ref{teo.parallel} in the Appendix, every parallel symmetric tensor on a 2-step nilpotent Lie group is left-invariant.

\begin{remark}\label{rem.multmetr}
Proposition \ref{pro.paralleltensors} in the Appendix shows that the eigenspaces of any parallel symmetric tensor induce a decomposition of $\mn$ in an orthogonal direct sum of ideals. Thus, if $\mn$ is irreducible,   the only parallel symmetric $2$-tensors are the multiples of the metric.
\end{remark}

Our first goal is to show that the decomposability of a left-invariant symmetric Killing tensor $S$ only depends on the decomposability of its component $S^\mv$ (which by Proposition \ref{pro.z} is also a Killing tensor).

\begin{lm} 
\label{lm.Senz} 
Every left-invariant symmetric tensor on $N$ defined by an element  $S\in \Sym^2\mz$ is a decomposable Killing $2$-tensor. 
\end{lm}
\begin{proof} For $z,z'\in \mz$, the symmetric product of the Killing vectors they define is $\Omega_{\xi_z\cdot\xi_{z'}}=z\cdot z'$, by \eqref{eq.xi}. According to \eqref{eq.S}, if $\{z_1, \ldots, z_n\}$ an orthonormal basis  of $\mz$, then  $S=\frac12 \sum_{1\leq i, j\leq n} \lela S z_i,z_j\rira \Omega_{\xi_{z_i}\cdot\xi_{z_j}}$.
\end{proof}

\begin{lm}\label{lm.jv}
Every left-invariant symmetric Killing $2$-tensor defined by an element $S\in \mz\cdot\mv$ is decomposable.
\end{lm}

\begin{proof} Let $\{z_1, \ldots,z_n\}$ be an orthonormal basis of $\mz$. We claim that $S=\sum_{s=1}^n \Omega_{\xi_{z_s}\cdot\xi_{Sz_s}}$. By \eqref{eq.symmprodxy},
$\Omega_{\xi_{z_s}\cdot\xi_{Sz_s}}(w)= z_s\cdot Sz_s-z_s\cdot [w,Sz_s]$, and thus
$\sum_{s=1}^n \Omega_{\xi_{z_s}\cdot\xi_{Sz_s}}$ has no component in $\Sym^2\mv$. Moreover, for $z,z'\in\mz$ one has
\begin{eqnarray*}
\sum_{s=1}^n \Omega_{\xi_{z_s}\cdot\xi_{Sz_s}}(w)(z,z')&=&-\sum_{s=1}^n\lela z_s,z\rira\lela [w,Sz_s],z'\rira+\lela z_s,z'\rira\lela [w,Sz_s],z\rira\\
&=&-\lela [w,Sz],z'\rira-\lela [w,Sz'],z\rira=0,
\end{eqnarray*}
where the last equality holds because $S$ is Killing and thus satisfies  \eqref{eq.killtensor}.
Finally, consider $x\in\mv$ and $z\in\mz$, then
$$
\sum_{s=1}^n \Omega_{\xi_{z_s}\cdot\xi_{Sz_s}}(w)(x,z)=\sum_{s=1}^m\lela z_s,z\rira\lela Sz_s,x\rira=\lela Sx,z\rira,
$$ since $S$ is symmetric. Therefore, $S=\sum_{s=1}^m \Omega_{\xi_{z_s}\cdot \xi_{Sz_s}}$ as claimed.
\end{proof}

By Lemma \ref{lm.Senz} and Lemma \ref{lm.jv} we obtain directly the following:

\begin{pro}\label{pro.Senv}
Let $S\in\Sym^2\mn$ define a left-invariant symmetric Killing tensor, with components denoted by $S^\mv$, $S^{\rm m}$ and $S^\mz$ as in Definition \ref{def.m}. Then $S$ is decomposable if and only if $S^\mv$ is decomposable.
\end{pro}

Our next aim is to show that if $\mn$ is a reducible 2-step nilpotent Lie algebra, then left-invariant symmetric Killing tensors on $N$ are determined by the ones on the factors.

Suppose $\mn=\mn_1\oplus\mn_2$ as an orthogonal direct sum of ideals. Denote $N,N_1,N_2$ the simply connected nilpotent Lie groups with Lie algebras $\mn,\mn_1,\mn_2$, respectively, and consider the left-invariant metrics $g_i$ that $\bil$ induces on $N_i$ for $i=1,2$ (also identified with the corresponding scalar products in $\mn_i$). The fact that  $(\mn,g)=(\mn_1,g_1)\oplus(\mn_2,g_2)$ implies that  $(N,g)=(N_1,g_1)\times (N_2,g_2)$ as Riemannian manifolds.

\begin{pro}\label{pro.sina}  Let $S\in \Sym^2\mn$ define a left-invariant symmetric Killing tensor on $N$. Let $S=S_1+S_{\rm nd}+S_2$ be its decomposition described in \eqref{eq.n12}. Then $S$ is  decomposable if and only if the left-invariant tensors defined by $S_i$ are decomposable symmetric Killing tensors on $N_i$ for $i=1,2$. 
\end{pro}
\begin{proof}

Let $S\in\Sym^2\mn$ and write $S=S_1+S_{\rm nd}+S_2$ as in the hypothesis. Suppose that $S$ is a decomposable Killing tensor.  By Proposition \ref{pro.idealsn}, $S_1,S_2,S_{\rm nd}$ are Killing tensors and $S_{\rm nd}\in \Sym^2{\mz}$; in particular, $S_{\rm nd}$ is decomposable by Lemma \ref{lm.Senz}. 
The fact that $S$ is decomposable in $N$ implies 
\begin{equation}\label{eq.OS}
S=\sum_{l=1}^tT_l+\sum_{1\leq j,k\leq m} a_{j,k}\Omega_{\xi_j\cdot \xi_k},\end{equation}
where $\{\xi_j\}_{j=1}^m$ is a basis of Killing vector fields of $N$, $a_{j,k}\in \R$ for $1\leq j,k\leq m$, and $T_l$ are parallel symmetric tensors on $N$ for $l=1, \ldots, t$. 

For each $j=1, \ldots, m$ and $i=1,2$, the Killing vector field $\xi_j$  defines a Killing vector field $\xi_j^i$ on $(N_i,g_i)$,  by restricting $\xi_j$ to $N_i\times \{e\}$ and projecting it to $TN_i$. 
In addition, since $TN_i$ is a left-invariant distribution in $N$, we have 
that $\Omega_{\xi_j^i}:\mn_i\lra \mn_i$ is given by $\Omega_{\xi_j^i}=\pr_{\mn_i}(\Omega_{\xi_j}|_{\mn_i})$. Moreover,  $\pr_{\mn_i}(\Omega_{\xi_j}|_{\mn_i})\cdot \pr_{\mn_i}(\Omega_{\xi_k}|_{\mn_i})=\pr_{\Sym^2\mn_i}(\Omega_{\xi_j\cdot \xi_k}|_{\mn_i})$, for all $1\leq j,k\leq m$, where the last projection is with respect to the decomposition \eqref{eq.n12}. Therefore $\Omega_{ \xi_j^i\cdot \xi_k^i}=\pr_{\Sym^2\mn_i}(\Omega_{\xi_j\cdot \xi_l}|_{\mn_i})$. 
Similarly, $T_l$ defines parallel symmetric tensors $T_l^i$ on $N_i$, for $l=1, \ldots,t$ and $i=1,2$. Therefore, restricting and projecting over $N_i$, for $i=1,2$, in \eqref{eq.OS} we have
$$S_i=\pr_{\Sym^2\mn_i} (S|_{\mn_i})=\sum_lT^i_l+\sum_{j,k} a_{j,k} \Omega_{\xi_j^i\cdot \xi_k^i},
$$
 thus $S_i$ is decomposable in $\mn_i$, for $i=1,2$.

Conversely, fix $i\in \{1,2\}$ and suppose $S_i$ is a decomposable left-invariant symmetric Killing tensor on $N_i$. Then $S_i=\sum_{l=1}^{t_i}T_l+\sum_{1\leq j,k\leq m_i} a_{j,k} \Omega_{\zeta_j\cdot \zeta_k}$ where $T_l$ are  parallel symmetric tensors on $N_i$ for $l=1, \ldots, t_i$,  $a_{j,k}\in \R$ for $1\leq j,k\leq m_i$ and  $\{\zeta_j\}_{j=1}^{t_i}$ is a basis of Killing vector fields of $N_i$. It is easy to check that both $T_l$ and $\zeta_j$ extend to parallel tensors and Killing vector fields of $N$, respectively, so that $S_i$ is decomposable as  a symmetric tensor in $N$. Hence $S=S_1+S_{\rm nd}+S_2$ is decomposable in $N$ by Lemma \ref{lm.Senz}. 
\end{proof}

Any 2-step nilpotent metric Lie algebra can be written as an orthogonal direct sum of ideals $\mn=\ma\oplus\mn_1\oplus\cdots\oplus \mn_s$, where $\ma$ is the abelian ideal in Remark \ref{rem.j} \eqref{it.remj}, and each $\mn_i$ is an irreducible 2-step nilpotent metric Lie algebra. The proposition above shows that the study of (in)decomposable left-invariant symmetric Killing tensors defined by elements in  $\Sym^2\mn$ can be reduced to the study of these objects defined by elements in $\Sym^2\ma$ and in each $\Sym^2\mn_i$. Notice that any left-invariant symmetric tensor defined by an element in  $\Sym^2\ma$ is parallel, thus Killing and decomposable. Therefore, it is sufficient to study (in)decomposable left-invariant symmetric Killing tensors on de Rham irreducible nilpotent Lie groups. Moreover, by Proposition \ref{pro.Senv} we might reduce our study to the decomposability of left-invariant symmetric Killing tensors defined by elements in $\Sym^2\mv$.

\medskip

For the rest of this section we assume that $(\mn,g)$ is an irreducible 2-step nilpotent metric Lie algebra. In particular, by Remark \ref{rem.j} \eqref{it.remj}, the map $j:\mz\lra\so(\mv)$ is injective.

\begin{lm} \label{lm.jinyecderiv} Let $T:\mv\lra\mv$ be a skew-symmetric endomorphism. Then $T$ admits at most one extension to a skew-symmetric derivation of $\mn$.
\end{lm}
\begin{proof} If $D\in \Dera(\mn)$ extends $T$, then for each $z\in \mz$,  $Dz$ is uniquely determined by the condition $j(Dz)=[T,j(z)]$ (see  \eqref{eq.deriv}), as $j$ is injective. 
\end{proof}

\begin{pro}\label{pro.decompv}
Let $S\in \Sym^2\mv$ define a left-invariant symmetric Killing tensor. Denote by $\lambda_i\in \R$ $(i=1,\ldots,k)$ the eigenvalues of $S$ in $\mv$ and  by $\mv_i$ the corresponding eigenspaces, so that $\mv$ decomposes as an orthogonal direct sum of $\mv=\bigoplus_{i=1}^k\mv_i$. Then $[\mv_i,\mv_j]=0$ for all $i\neq j$ and $j(z)$ preserves $\mv_i$ for all $z\in \mz$, $i=1, \ldots, k$.
\end{pro}
\begin{proof}
Let $x\in\mv_i$, $y\in\mv_j$ be eigenvectors of $S$ associated to the eigenvalues $\lambda_i\neq \lambda_j$. Then by \eqref{eq.killtensor}, $[Sx,y]=[x,Sy]$, which leads to $\lambda_i[x,y]=\lambda_j[x,y]$. Therefore $[\mv_i,\mv_j]=0$. This implies, by \eqref{eq:jota} that $j(z)$ preserves each $\mv_i$ for every $z\in\mz$.
\end{proof}

Before studying the general case, let us notice that the case where $S\in \Sym^2\mv$ has a unique eigenvalue is trivial:

\begin{pro} The restriction $g|_\mv$ of the metric $g$ to $\mv$ defines left-invariant decomposable symmetric Killing tensor on $N$. 
\end{pro}

\begin{proof} Let $\{e_1, \ldots, e_m\}$ and $\{z_1,\ldots,z_n\}$ be  orthonormal basis of $\mv$ and $\mz$, respectively. Then
$$ g|_\mv= \frac12\sum_{i=1}^{m} e_i\cdot e_i =\bil - \frac{1}2  \sum_{s=1}^n z_s\cdot z_s=\bil - \frac{1}2  \sum_{s=1}^n\xi_{ z_s}\cdot \xi_{z_s},$$
where $\xi_{z_s}$ denotes the Killing vector field defined by $z_s$ on $N$ (which is both left- and right-invariant). Therefore, $g|_\mv$ defines a decomposable Killing tensor.
\end{proof}

\noindent {\bf Notation.} Let $S\in\Sym^2\mv$ define a left-invariant symmetric Killing tensor on $N$ and let $\mv=\mv_1\oplus \cdots\oplus\mv_k$ be the decomposition of $\mv$ as in Proposition \ref{pro.decompv}. Denote $\pr_{\mv_i} :\mv\lra \mv$ the orthogonal projection onto $\mv_i$.

For $\lambda_1, \cdots, \lambda_k\in \R$ denote 
$\pmb{\lambda}=(\lambda_1, \ldots,\lambda_k)$. Given $a\in \R$, we define $$\pmb\lambda-a :=(\lambda_1-a, \ldots, \lambda_k-a).$$
Let $\Pi^{\pmb\lambda}$ be the endomorphism of $\mv$ defined as $\Pi^{\pmb\lambda}=\sum_{i=1}^k\lambda_i \pr_{\mv_i}$. For each $z\in \mz$ we shall denote by $T_z^{\pmb\lambda}$ the skew-symmetric endomorphism of $\mv$ given by
$$T^{\pmb\lambda}_z=j(z)\circ\Pi^{\pmb\lambda}.$$

We are now ready for the main result of this section:

\begin{teo} \label{pro.necT} Let $\mn=\mv\oplus\mz$ be an irreducible $2$-step nilpotent Lie algebra. Let $S\in Sym^2\mv$ define a left-invariant  symmetric Killing tensor. Denote the eigenvalues of $S$ (viewed as endomorphism of $\mv$) by $\lambda_1, \ldots, \lambda_k$. Then $S$ is decomposable if and only if there exists $a\in\R$ such that for every $z\in\mz$, $T^{\pmb\lambda-a}_z$ extends to a skew-symmetric derivation of $\mn$.
\end{teo}

\begin{proof} 
Let $\{D_1,\ldots ,D_d\}$ be a basis of $\Dera(\mn)$ and let $\{z_1, \ldots, z_n\}$ be an orthonormal basis of $\mz$. Let $\{e_1, \ldots, e_m\}$ be an orthonormal basis of $\mv$ adapted to the decomposition $\mv=\mv_1\oplus \cdots\oplus\mv_k$ of eigenspaces of $S$, in the sense that there exist $m_0, m_1,\ldots,m_k\in \N$ such that $e_{m_{i-1}+1},\ldots ,e_{m_{i}}$ spans $\mv_i$ for every $i=1, \ldots,k$; in particular, $m_0=0$ and $m_k=m$. By \eqref{eq.S} we then have $S=\frac{1}2\sum_{i=1}^{k}\sum_{l=m_{i-1}+1}^{m_{i}}\lambda_i\  e_l\cdot e_l$. 

By Remark \ref{rem.multmetr}, if $S$ is \nameend, then there exist constants  $a,\ a_{i,l},\ b_{i,t},\ c_{s,t},\ d_{q,r}\in\R$ and $D'_s, D''_i\in \Dera(\mn)$ for $i,l\in\{1,\ldots,m\}$, $s,t\in\{1,\ldots,n\}$, and $q,r\in\{1,\ldots,d\}$ with $a_{i,l}=a_{l,i}$, $c_{s,t}=c_{t,s}$ and $d_{q,r}=d_{r,q}$, such that
\begin{eqnarray}
S&=& a \bil +\sum_{1\leq i, l\leq m} a_{i,l}\Omega_{ \xi_{e_i}\cdot\xi_{e_l}}+\sum_{\begin{smallmatrix}1\leq i\leq m\nonumber\\
1\leq t\leq n
\end{smallmatrix}} b_{i,t} \Omega_{\xi_{e_i}\cdot\xi_{z_t}}+\sum_{1\leq s,t\leq n} c_{s,t}\Omega_{\xi_{z_s}\cdot\xi_{z_t}}
\\
&&+\sum_{s=1}^n \Omega_{\xi_{D'_s}\cdot\xi_{z_s}} 
+\sum_{i=1}^{m} \Omega_{\xi_{D''_i}\cdot\xi_{e_i}} +\sum_{1\leq q,r\leq d}d_{q,r} \Omega_{\xi_{D_q}\cdot\xi_{D_r}}. \label{eq.combkill}
\end{eqnarray}

Evaluating this equation at $w=0$ and using \eqref{eq.xi} we obtain $b_{i,t}=0$ for all $i,t$, $c_{s,t}=0$ for $s\ne t$, $a_{i,l}=0$ for $i\ne l$, $c_{s,s}=-\frac{a}2$, and $a_{l,l}=\frac12 \lambda_i-\frac{a}2$ for each $i\in\{1,\ldots,k\}$ and $m_{i-1}+1\leq l\leq m_i$. Consequently, \eqref{eq.combkill} becomes
\begin{eqnarray*}
S&=&ag + \sum_{i=1}^k\sum_{l=m_{i-1}+1}^{m_i}
\frac12(\lambda_i-a) \Omega_{\xi_{e_l}\cdot\xi_{e_l}}-\frac{a}2\sum_{s=1}^{n} \Omega_{\xi_{z_s}\cdot\xi_{z_s}}+\sum_{s=1}^n \Omega_{\xi_{D'_s}\cdot\xi_{z_s}} 
+\sum_{i=1}^{m} \Omega_{\xi_{D''_i}\cdot\xi_{e_i}} \\
&&\quad +\sum_{1\leq q,r\leq d}d_{q,r} \Omega_{\xi_{D_q}\cdot\xi_{D_r}}.
\end{eqnarray*}
We replace by the formulas in \eqref{eq.symmprodxy}--\eqref{eq.symmprodxD} and we obtain, for all $w\in \mn$, 
\begin{eqnarray*}
S&=&\frac12\sum_{i=1}^k\sum_{l=m_{i-1}+1}^{m_i}
\lambda_ie_l\cdot e_l
+\sum_{i=1}^k\sum_{l=m_{i-1}+1}^{m_i}\frac12(\lambda_i-a)
 \left( -2\ e_l\cdot[w,e_l]+[w,e_l]^2\right) \\
&&+ \sum_{s=1}^n \left(z_s\cdot D'_sw -\frac12 z_s\cdot [w,D'_sw]\right)+\sum_{1\leq q,r\leq d}d_{q,r} \Omega_{\xi_{D_q}\cdot\xi_{D_r}} \\
&&+ \sum_{i=1}^{m} e_i\cdot D''_iw -([w,e_i]\cdot D''_iw+\frac12 e_i\cdot [w,D''_iw])+\frac12 [w,e_i]\cdot [w,D''_iw]
\end{eqnarray*}
The right hand side is polynomial in the coordinates of $w$ (expressed in the above basis of $\mn$), whereas the left hand side is constant. The equality above holds if and only if the coefficients in each degree coincide. Equality in degree one holds if and only if
$$\sum_{s=1}^n  z_s\cdot D'_sw + \sum_{i=1}^{m} e_i\cdot D''_iw=\sum_{i=1}^k\sum_{l=m_{i-1}+1}^{m_i}(\lambda_i-a) e_l\cdot[w,e_l].$$ 
Contracting this equality with $z_{r}$ for some fixed $r$ in $\{1,\ldots, n\}$ gives for any $w\in\mv$:
\begin{eqnarray*}
D'_{r}w &=&\sum_{i=1}^k\sum_{l=m_{i-1}+1}^{m_i}(\lambda_i-a)   \lela [w,e_l],z_{r}\rira e_l=\sum_{i=1}^k\sum_{l=m_{i-1}+1}^{m_i}(\lambda_i-a)   \lela j(z_{r})w,e_l\rira e_l\\
&=& j(z_r)\sum_{i=1}^k(\lambda_i-a) \pr_{\mv_i}w=T_{z_r}^{\pmb\lambda-a}w.
\end{eqnarray*} 
Therefore, $D_r'$ extends $T_{z_r}^{\pmb\lambda-a}$ to a skew-symmetric derivation of $\mn$. Since the map $z\mapsto T^{\pmb\lambda-a}_z$ is linear, we obtain that $T^{\pmb\lambda-a}_z$ extends to a skew-symmetric derivation for all $z\in \mz$.

\smallskip

In order to prove the converse statement, let $S\in \Sym^2\mv$ define a left-invariant symmetric Killing tensor with eigenvalues $\lambda_1,\ldots,\lambda_k$ and let $a\in \R$ be such that for every $z\in\mz$, $T_z^{\pmb\lambda-a}$ extends to a derivation $D_z$ of $\mn$, where $\pmb\lambda=(\lambda_1,\ldots,\lambda_k)$. We shall explicit a linear combination of the metric and products of Killing vector fields as in \eqref{eq.combkill} that gives $S$.

We claim that \eqref{eq.combkill} is satisfied for the following choice of coefficients and derivations: $a$ is the one coming from the hypothesis, $b_{s,t}=0$ for all $s,t$, $c_{s,t}=0$ for $s\neq t$,  $c_{s,s}=-\frac{a}2$ for $s=1,\ldots n$, $a_{i,l}=0$ for $i\neq l$ and $\frac{a}2+a_{l,l}=\frac12 \lambda_i$ for each $i=1,\ldots,k$ and $m_{i-1}+1\leq l\leq m_i$, $d_{q,r}=0$ for all $q,r$, and $D''_i=0$ for all $i$, and for $s=1,\ldots, n$, $D'_s$ is the skew-symmetric derivation extending $T_{z_s}^{\pmb\lambda-a}$. With these choices, we define the decomposable Killing tensor
\begin{eqnarray}\label{eq.converse}
K&:=&ag + \sum_{i=1}^k\sum_{l=m_{i-1}+1}^{m_i}
\frac12(\lambda_i-a) \Omega_{\xi_{e_l}\cdot\xi_{e_l}}-\frac{a}2\sum_{s=1}^{n} \Omega_{\xi_{z_s}\cdot\xi_{z_s}}+\sum_{s=1}^n \Omega_{\xi_{D'_s},\xi_{z_s}}=\nonumber \\
&=&\frac12\sum_{i=1}^k\sum_{l=m_{i-1}+1}^{m_i}
\lambda_ie_l\cdot e_l
+\sum_{i=1}^k\sum_{l=m_{i-1}+1}^{m_i}\frac12(\lambda_i-a)
 \left( -2\ e_l\cdot[w,e_l]+[w,e_l]^2\right) \\ 
&&+ \sum_{s=1}^n \left(z_s\cdot D'_sw -\frac12 z_s\cdot [w,D'_sw]\right),\nonumber 
\end{eqnarray}
where the last equality follows again from \eqref{eq.symmprodxy}--\eqref{eq.symmprodxD}, and we will show that actually $K=S$. 

A reasoning as before gives $\sum_{i=1}^k\sum_{l=m_{i-1}+1}^{m_i}(\lambda_i-a) e_l\cdot[w,e_l] = \sum_{s=1}^n z_s\cdot D'_sw $ for all $w\in\mv$. We claim that this also holds for $w\in\mz$. Since the left hand side vanishes in this case, we need to show that $\sum_{s=1}^n z_s\cdot D'_sw=0$ for all $w\in\mz$.

For every $z,z' \in\mz$, $j(D_zz')$, $j(D_{z'}z)$ are both in $\so(\mv)$, they preserve each $\mv_i$ and actually they are opposite in each $\mv_i$ since by \eqref{eq.deriv}:
$$j(D_zz')|_{\mv_i}=(\lambda_i-a)[j(z)|_{\mv_i},j(z')|_{\mv_i}]=-(\lambda_i-a)[j(z')|_{\mv_i},j(z)|_{\mv_i}]=-j(D_{z'}z)|_{\mv_i}.$$ 
Therefore $j(D_zz')=-j(D_{z'}z)$ and thus $D_zz'=-D_{z'}z$ since $j$ is injective. 

Now given $w\in \mz$, 
\begin{eqnarray*}
\sum_{s=1}^n (z_s\cdot D'_sw)(z,z')&=&\sum_{s=1}^n \lela z_s,z\rira\lela D_{z_s}w,z'\rira+ \lela z_s,z'\rira\lela D_{z_s}w,z\rira\\
&=&-\sum_{s=1}^n \lela z_s,z\rira\lela D_w z_s,z'\rira+ \lela z_s,z'\rira\lela D_w z_s,z\rira\\
&=&-\lela D_w z,z'\rira- \lela D_w z',z\rira=0, 
\end{eqnarray*} for all $z,z'\in\mz$, thus proving our claim. From \eqref{eq.converse} we thus get
\begin{equation}\label{eq.ks}K=S+\sum_{i=1}^k\sum_{l=m_{i-1}+1}^{m_i}\frac12(\lambda_i-a)[w,e_l]^2 -\frac12\sum_{s=1}^n  z_s\cdot [w,D'_sw].\end{equation}

Finally, for every $w\in\mv$ we have:
\begin{eqnarray*}\sum_{s=1}^n  z_s\cdot [w,D'_sw]&=&\sum_{s=1}^n\sum_{i=1}^k\sum_{l=m_{i-1}+1}^{m_i}  z_s\cdot [w,e_l]\ g(e_l,D'_sw)\\ 
&=&-\sum_{s=1}^n\sum_{i=1}^k\sum_{l=m_{i-1}+1}^{m_i}  z_s\cdot [w,e_l]\ g(T_{z_s}^{\pmb\lambda-a}e_l,w)\\
&=&-\sum_{s=1}^n\sum_{i=1}^k\sum_{l=m_{i-1}+1}^{m_i}(\lambda_i-a)  z_s\cdot [w,e_l]\ g(j(z_s)e_l,w)\\
&=&-\sum_{s=1}^n\sum_{i=1}^k\sum_{l=m_{i-1}+1}^{m_i}  (\lambda_i-a) z_s\cdot [w,e_l]\ g(z_s,[e_l,w])\\&=&\sum_{i=1}^k\sum_{l=m_{i-1}+1}^{m_i}  (\lambda_i-a)  [w,e_l]^2,
 \end{eqnarray*}
which together with \eqref{eq.ks} shows that $S=K$, so $S$ is decomposable.
\end{proof}

For the statement of the next result we introduce some terminology. We say that an orthogonal decomposition $\mv=\mv_1\oplus\ldots\oplus\mv_k$ is $j(\mz)$-invariant if it is preserved by $j(z)$ for every $z\in\mz$. Equivalently, such a decomposition is $j(\mz)$-invariant if $[\mv_i,\mv_j]=0$ for all $i\ne j$.
In this case, we denote by $j(\mz)|_{\mv_i}$ the vector subspace of $\so(\mv_i)$ 
$$j(\mz)|_{\mv_i}:=\{j(z)|_{\mv_i}:\mv_i\lra \mv_i\ |\ z\in\mz\}.$$

As a first application of Theorem \ref{pro.necT}, we describe Lie groups for which every left-invariant symmetric Killing $2$-tensor is decomposable.

\begin{cor}\label{pro.condz} Let $\mn$ be an irreducible $2$-step nilpotent Lie algebra such that one of the following conditions holds:
\begin{itemize}
\item $\dim \mz=1$;
\item $\dim \mz\geq 2$ and for any $j(\mz)$-invariant orthogonal decomposition $\mv=\mv_1\oplus\ldots\oplus\mv_k$, one has that $j(\mz)|_{\mv_i}$ is an abelian subalgebra of $\so(\mv_i)$ for all $i=1,\ldots,k$ except possibly for one index $i_0$.
\end{itemize}
Then any left-invariant symmetric Killing $2$-tensor on $N$ is decomposable.
\end{cor}

\begin{proof} 
In view of Proposition \ref{pro.Senv} it is enough to show that any 
left-invariant symmetric Killing tensor defined by al element $\Sym^2\mv$ is decomposable. Let $S\in \Sym^2\mv$  and let $\mv=\mv_1\oplus\ldots\oplus\mv_k$ be the orthogonal decomposition in the eigenspaces of $S$ in $\mv$. We apply the converse of Theorem \ref{pro.necT} to show that $S$ is decomposable.

If $\dim \mz=1$, then $j(\mz)$ is an abelian subalgebra of $\so(\mv)$. Choose $a\in \R$ arbitrary and let $D_z$ be the endomorphism of $\mn$ which extends $T_z^{\pmb\lambda-a}$ by zero, for $z\in \mz$. Then, for all $z'\in\mz$,
$j(D_zz')=0$ by definition, and 
$$[D_z|_\mv,j(z')]=[T_z^{\pmb\lambda-a},j(z')]=(\lambda-a)[j(z),j(z')]=0$$
because $\mz$ is of dimension 1. Hence \eqref{eq.deriv} holds and thus $D_z$ is a skew-symmetric derivation of $\mn$ which extends $T_z^{\pmb\lambda-a}$.

For the second case, up to reordering the indexes, one can assume that $j(\mz)|_{\mv_i}$ is an abelian subalgebra of $\so(\mv_i)$ for all $i=2,\ldots,k$. For each $z\in \mz$, let $D_z$ be the extension by zero of $T_z^{\pmb\lambda-\lambda_1}$. We claim that $D_z$ is a derivation of $\mn$ for all $z$. Indeed, $[T_z^{\pmb\lambda-\lambda_1},j(z')]|_{\mv_1}=0$, and for $i\geq 2$ we have $$[T_z^{\pmb\lambda-\lambda_1},j(z')]|_{\mv_i}=(\lambda_i-\lambda_1)[j(z)|_{\mv_i},j(z')|_{\mv_i}]=0$$ since 
$j(\mz)|_{\mv_i}$ is an abelian subalgebra of $\so(\mv_i)$. Therefore $D_z$ satisfies \eqref{eq.deriv}, so it is a derivation of $\mn$. 

In both cases we thus obtain that $S$ is decomposable by Theorem  \ref{pro.necT}.
\end{proof}

\begin{ex}
Every left-invariant symmetric Killing 2-tensor on the Heisenberg Lie group in Example \ref{ex.heis} is decomposable. Indeed, the center of $\mh_{2n+1}$ is one dimensional.
\end{ex}

\begin{teo}\label{teo.dimleq7} Every left-invariant symmetric Killing $2$-tensor on a $2$-step nilpotent Lie group of dimension $\leq 7$ is decomposable.
\end{teo}

\begin{proof} Let $N$ be a 2-step nilpotent Lie group and suppose there is some left-invariant symmetric Killing $2$-tensor on $N$ which is indecomposable. By Proposition \ref{pro.sina} we can assume that its Lie algebra $\mn$ is irreducible. Corollary \ref{pro.condz} implies that $\dim \mz\geq 2$ and there exists a decomposition $\mv=\mv_1\oplus\cdots\oplus\mv_k$ in $j(z)$ invariant subspaces such that for at least two indices $i_0$, $i_1$, the vector subspaces $j(\mz)|_{\mv_{i_0}}$ and $j(\mz)|_{\mv_{i_1}}$ are not an abelian subalgebras of $\so(\mv_{i_0})$ and $\so(\mv_{i_1})$ respectively. This, in particular, implies that $\dim \mv_{i_0}$ and $\dim \mv_{i_1}$ are both $\geq3$ and therefore $\dim \mn\geq \dim(\mz\oplus\mv_{i_0}\oplus\mv_{i_1})\geq 8$. 
\end{proof}

A second application of Theorem \ref{pro.necT} is a method to construct 2-step nilpotent Lie groups admitting left-invariant indecomposable symmetric Killing $2$-tensors. 

\begin{cor} \label{cor.ind} Let $S\in\Sym^2\mv$ define a left-invariant symmetric Killing tensor and consider the decomposition of $\mv$ as orthogonal direct sum of eigenspaces of $S$, $\mv=\bigoplus_{i=1}^k\mv_i$. If there exist $i\neq j$ such that $j(\mz)|_{\mv_i}$ and $j(\mz)|_{\mv_j}$ are not subalgebras of $\so(\mv_i)$ and $\so(\mv_j)$, respectively, then $S$ is indecomposable.
\end{cor}

\begin{proof} Suppose that $S$ is decomposable and denote by $\pmb\lambda$ the vector of distinct eigenvalues of $S$ viewed as endomorphism of $\mv$ and by $\mv=\mv_1\oplus\cdots\oplus\mv_k$ the corresponding decomposition in eigenspaces of $S$. By Theorem \ref{pro.necT} there exists $a\in \R$ such that for all $z\in \mz$,  $T_{z}^{\pmb\lambda-a}$ extends to a skew-symmetric derivation of $\mn$. Since $i\neq j$, the eigenvalues $\lambda_i$ and $\lambda_j$ are distinct, so we may assume that $(\lambda_{i}-a)\neq 0$ (otherwise just replace $i$ by $j$).

Given $z\in \mz$ denote $D_z$ the skew-symmetric derivation of $\mn$ extending $T_{z}^{\pmb\lambda-a}$. Then,  \eqref{eq.deriv} gives
\begin{equation}\label{eq.jtau}
j(D_z z') = [T_{z}^{\pmb\lambda-a},j(z')], \mbox{ for all } z,z'\in \mz.
\end{equation}
Restricting to $\mv_{i}$, this equality gives $j(D_z z')|_{\mv_{i}} = (\lambda_{i}-a)[j(z)|_{\mv_{i}},j(z')|_{\mv_{i}}]$ which implies that $j(\mz)|_{\mv_{i}}$ is a Lie subalgebra of $\so(\mv_{i})$, since $\lambda_{i}-a\neq 0$. This contradicts the hypothesis, thus showing that $S$ is indecomposable.
\end{proof}

Next we present the aforementioned construction method. Let $(V,\bil)$ be an inner product space and consider a vector subspace $\mz\subset \so(V,\bil)$ which is not a Lie subalgebra of $\so(V,\bil)$. Set $\mv_1=\mv_2:=V$, $\mn:=\mv_1\oplus\mv_2\oplus\mz$ and define an inner product on $\mn$ making $\mv_1$, $\mv_2$, $\mz$ all orthogonal and extending $\bil$ in $\mv_1=\mv_2=V$ (the choice is arbitrary in $\mz$). On $\mn$ we define the Lie bracket such that
$$[\mv_1\oplus\mv_2,\mz]=0, \quad [\mv_1,\mv_2]=0, \quad \lela [x_i,y_i], z\rira=\lela z(x_i),y_i\rira, \mbox{ for }x_i,y_i\in \mv_i. $$
Then the map $j:\mz\lra \so(\mv_1\oplus \mv_2)$ is given by $j(z)=\begin{pmatrix} z&0\\0&z\end{pmatrix}$.

Fix $\alpha\neq 1$ and consider the endomorphism $S={\rm Id}|_{\mv_1}+ \alpha {\rm Id}|_{\mv_2}$ of $\mv$. By Proposition \ref{pro.caracterizKilling}, $S$ defines a left-invariant symmetric Killing $2$-tensor. Moreover, since $j(\mz)|_{\mv_1}$ and $j(\mz)|_{\mv_2}$ (which are both isomorphic to $\mz$) are not Lie subalgebras of $\so(\mv_1)=\so(\mv_2)$, Corollary \ref{cor.ind} shows that $S$ is indecomposable. 

The following is an explicit example of this construction.

\begin{ex} Let $\mn$ be the 2-step nilpotent Lie algebra of dimension 8 with orthonormal basis $e_1, \ldots, e_6, z_1,z_2$ satisfying the bracket relations
$$[e_1,e_2]=z_1=[e_4,e_5], \qquad [e_2,e_3]=z_2=[e_5,e_6].$$ 
Consider $S\in \Sym^2\mn$ defined by $Se_i=e_i$ for $i=1, 2,3$, $Se_i=2e_i$ for $i=4,5,6$, and $Sz_1=Sz_2=0$. Since $S$ verifies \eqref{eq.killtensor} it defines a left-invariant Killing tensor. The decomposition of $\mv$ is given by $\mv_1={\rm span}\{e_1,e_2,e_3\}$ and $\mv_2={\rm span}\{e_3,e_4,e_6\}$. For both $i=1,2$, $j(\mz)|_{\mv_i}$, in the corresponding basis, is spanned by the matrices
$$\left(\begin{array}{ccc}
0 &-1&0\\
1&0&0\\
0&0&0
\end{array} \right),\quad \left(\begin{array}{ccc}
0 &0&0\\
0&0&-1\\
0&1&0
\end{array} \right).$$
The vector space spanned by these two matrices is clearly not a Lie subalgebra of $\so(3)$, so the left-invariant symmetric tensor defined by $S$ is indecomposable.
\end{ex}

\begin{remark} The above examples of indecomposable symmetric Killing tensors can be used to produce examples on compact nilmanifolds.

Recall that a nilpotent Lie group $N$ admits a co-compact discrete subgroup $\Gamma$ if and only if its Lie algebra $\mn$ is rational, that is, $\mn$ admits a basis with structure constants in $\Q$ (cf. \cite{RA}). In this case, any left-invariant Riemannian metric on $N$ defines a Riemannian metric on the compact manifold $\Gamma\bs N$ so that the natural projection $p:N\lra \Gamma\bs N$ is a locally isometric covering. 

Every left-invariant symmetric Killing tensor field $S$ on $N$ projects to a Killing tensor field $\bar S$ on $\Gamma\bs N$ and clearly, if $S$ is indecomposable on $N$, then $\bar S$ is indecomposable on $\Gamma\bs N$.
\end{remark}

\appendix

\section{Parallel distributions on nilpotent Lie groups}

In this section we will prove a property of parallel tensors on nilpotent Lie groups which was used in Theorem \ref{pro.necT}, but which, we think, is also of independent interest.

\begin{teo}\label{teo.parallel}
Every parallel symmetric $2$-tensor on a $2$-step nilpotent Lie group with left-invariant Riemannian metric $(N,\bil)$ is left-invariant.
\end{teo}

\begin{proof} Since the Lie algebra $\mn$ is $2$-step nilpotent, we have $[[\mn,\mn],\mn]=0$, so the derived algebra $\mn':=[\mn,\mn]$ is contained in the center $\mz$ of $\mn$. Consider first the case where $\mn'=\mz$. It is easy to check that the
Ricci tensor of $(N,\bil)$ at the identity preserves the decomposition
$\mn=\mv\oplus\mz$ (where we recall that $\mv=\mz^\perp$) and is positive definite on $\mz$ and negative definite on $\mv$ \cite[Proposition 2.5]{EB}.
In particular $(N,\bil)$ carries no parallel vector fields, so its factors 
$(N_i,g_i)$, $i=1,\ldots,k$ in the de Rham decomposition are non-flat irreducible
Riemannian manifolds. Correspondingly, we write $\mn=T_e N=T_1\oplus\ldots\oplus
T_k$.

Let $D$ be any parallel distribution on $N$. We claim that $D_e$ (which of course
determines $D$) is the direct sum of some of the subspaces $T_i$. To see this,
consider the holonomy group $\mathrm{Hol}_e(N)\subset \mathfrak{so}(\mn)$. By the de
Rham theorem, $\mathrm{Hol}_e(N)=H_1\times\ldots\times H_k$ where $H_i$ acts
irreducibly on $T_i$ for each $i=1,\ldots,k$. Since $D$ is parallel, $D_e$ is a
subspace of $T_1\oplus\ldots\oplus T_k$ invariant by $H_1\times\ldots\times H_k$.
The orthogonal projection denoted by $D_i$  of $D$ onto $T_i$ obviously contains
$D_e\cap T_i$ and both are $H_i$-invariant subspaces of $T_i$. Thus $D_i=0$ or
$D_i=T_i$ for every $i$. Since 
$$\bigoplus_i(D_e\cap T_i)\subset D_e\subset \bigoplus_i D_i,$$
to prove our claim, we need to show that $D_e\cap T_j=D_j$ for every $j$. By the
irreducibility of $T_j$ as $H_j$-representation, it is enough to show that if
$D_j\ne 0$ then $D_e\cap T_j\ne 0$.

Let us fix any $j\in\{1,\ldots,k\}$ with $D_j\ne 0$, and take any non-zero $x_j\in
D_j$. Then there exist $x_i\in D_i$ for every $i\ne j$ such that
$x:=x_1+\ldots+x_k\in D$. Every element $h_j\in H_j$ acts trivially on $T_i$ for
$i\ne j$, so we get that $D_e\ni
h_j(x)=x_1+\ldots+x_{j-1}+h_j(x_j)+x_{j+1}+\ldots+x_k$, and thus
$x_j-h_j(x_j)=x-h_j(x)\in D_e$. 

Since $H_j$ is not trivial and $T_j$ is an irreducible $H_j$-representation, there
exists $h_j\in H_j$ such that $h_j(x_j)\ne x_j$. The non-zero vector $x_j-h_j(x_j)$
belongs to $D_e\cap T_j$, thus proving our claim.

In particular, this proves that the Riemannian manifold $(N,\bil)$ carries only a
finite number of parallel distributions. Assume now that $K$ is a parallel symmetric
2-tensor on $N$. If $\lambda_1,\ldots,\lambda_l$ denote the spectrum of $K_e$, the
corresponding eigenspaces define parallel distributions $V_1,\ldots,V_l$ on $N$ such
that $K=\sum_{i=1}^l \lambda_i \mathrm{Id}_{V_i}$. Let us fix $i\in\{1,\ldots,l\}$.
For every element $h\in N$, $h_*(V_i)$ belongs to a finite set of parallel
distributions, so by continuity, as $N$ is connected, $h_*(V_i)=V_i$ is independent
on $h$. This shows that $h_*K=K$, so $K$ is left-invariant.

For the general case, one defines as in Remark \ref{rem.j} \eqref{it.remj} the abelian ideal
$\ma$ of $\mn$ by $\ma:=\mz\cap\mn'^{\perp}$ and observe that $\tilde
\mn:=\mv\oplus\mn'$ is also an ideal and $\mn=\tilde \mn\oplus\ma$ is an orthogonal
direct sum of ideals. Correspondingly, the simply connected Lie group $N$ endowed
with the left-invariant metric $\bil$ is a Riemannian product $(\tilde N,\tilde\bil)\times
\mathbb{R}^n$, where $n=\mathrm{dim}(\ma)$ and $\tilde\bil$ is the restriction of
$\bil$ to $\tilde\mn$. By construction, the derived algebra $\tilde \mn'$ is equal
to the center of $\tilde\mn$, so by the first part of the proof, every parallel
symmetric tensor on $\tilde N$ is left-invariant. Moreover, since $(\tilde
N,\tilde\bil)$ has no parallel vector field, it follows that every parallel
symmetric tensor on $(\tilde N,\tilde\bil)\times \mathbb{R}^n$ is the sum of a
constant symmetric tensor on $\mathbb{R}^n$ and a parallel symmetric tensor on
$\tilde N$, and thus is left-invariant. This concludes the proof.
\end{proof}

The hypothesis that $N$ is 2-step nilpotent is essential in the above theorem, as
shown by the following (counter-)example.

\begin{ex}Assume that $G$ is the simply connected Lie group generated by the
solvable $3$-dimensional Lie algebra $\mg$ defined by the brackets
$$[e_1,e_2]=0,\qquad [e_2,e_3]=e_1,\qquad [e_3,e_1]=e_2.$$
Consider the left-invariant metric on $G$ defined by requiring that $e_1,e_2,e_3$ is
orthonormal and denote by $E_i$ the left-invariant vector field on $G$ which is
equal to $e_i$ at the identity. The Koszul formula \eqref{eq.Koszul} immediately shows that the
Levi-Civita connection $\nabla$ of this metric is given by
$$\nabla_{E_1}E_1=\nabla_{E_2}E_2=\nabla_{E_3}E_3=\nabla_{E_1}E_2=\nabla_{E_2}E_1=\nabla_{E_1}E_3=\nabla_{E_2}E_3=0,$$
$$ \nabla_{E_3}E_1=E_2,\qquad \nabla_{E_3}E_2=-E_1.$$
A straightforward computation then shows that the Riemannian curvature of $\nabla$
vanishes, so every symmetric tensor on $\mg$ defines a parallel tensor on $G$ by
parallel transport. On the other hand, the left-invariant tensor $E_1\cdot E_1$ is
not parallel (since $\nabla_{E_3}(E_1\cdot E_1)=2E_1\cdot E_2$), so it does not
coincide with the parallel tensor on $G$ defined by the parallel transport of
$e_1\cdot e_1$.\end{ex}

We will now show that the eigendistributions of a parallel symmetric endomorphism of $TN$ (which are automatically left-invariant by Theorem \ref{teo.parallel}) define a decomposition of $\mn$ as orthogonal sum of ideals.

\begin{pro} \label{pro.paralleltensors} Let $S\in \Sym^2\mn$ define a left-invariant symmetric tensor on $N$. Denote by $\lambda_i\in \R$
$(i=1,\ldots, k)$ the eigenvalues of $S$ and by $\mn_i$ the corresponding
eigenspaces, so that $\mn$ decomposes in an orthogonal direct sum of
$\mn=\bigoplus_{i=1}^k\mn_i$. 
Then $\nabla S=0$ if and only if each $\mn_i$ is an ideal on $\mn$.
\end{pro}

\begin{proof} The ``if'' part is obvious, since if each $\mn_i$ is an ideal on
$\mn$, then $\mathrm{Id}_{\mn_i}$ is parallel as endomorphism of $\mn$, and thus so
is $S=\sum_{i=1}^k\lambda_i\mathrm{Id}_{\mn_i}$.

Conversely, if $S$ is parallel, then each $\mn_i$ defines a left-invariant parallel
distribution on $N$, so it is enough to show that if the left-invariant distribution on $N$ defined by a vector subspace $D\subset \mn$ is parallel, then $D$ is an ideal of $\mn$. 

Denote by $D^\mv\supset D\cap \mv$ the orthogonal projection of $D$ onto $\mv$. We
claim that $D^\mv= D\cap\mv$. Let $x\in D^\mv$ be some arbitrary element orthogonal
to $D\cap \mv$. Then there exists $z_0\in\mz$ with $x+z_0\in D$. For every $z\in\mz$
we have $D\ni \nabla_z(x+z_0)=-\frac12j(z)(x)$ (see \eqref{eq:nabla}), whence $j(z)(x)\in D\cap \mv$ for
every $z\in\mz$. Moreover, $\nabla_z(j(z)(x))=-\frac12j(z)^2(x)$ clearly belongs to
$D$ (as $D$ is parallel) and to $\mv$, so in particular, by the choice of $x$, it is
orthogonal to $x$. Since $j(z)$ is skew-symmetric, this shows that $j(z)(x)=0$ for
every $z\in\mz$, whence $x=0$, thus proving our claim. 

This shows that $D=(D\cap\mv)\oplus(D\cap\mz)$. In order to prove that $D$ is an
ideal it is now equivalent to check that for every $x\in D\cap \mv$ and $y\in \mv$
one has $[x,y]\in D$ (all other commutators are automatically $0$). This is clear
since $[x,y]=-2\nabla_y x\in D$ as $D$ is parallel.
\end{proof}

As a corollary, we get the following result about $2$-step nilpotent Lie algebras, which is also of independent interest.

\begin{cor} Let $(\mn,g)$ be a $2$-step nilpotent metric Lie algebra. Then there exist {\em irreducible} $2$-step nilpotent metric Lie algebras $(\mn_i,g_i)$ $i\in\{1,\ldots,k\}$ (unique up to reordering) such that 
$$(\mn,g)=\bigoplus_{i=1}^k(\mn_i,g_i)\oplus (\ma,g_0),$$
for some abelian metric Lie algebra $(\ma,g_0)$.
\end{cor}

\begin{proof} By definition, every reducible $2$-step nilpotent Lie algebra is the orthogonal direct sum of ideals, each of them being $2$-step nilpotent or abelian. The above decomposition thus always exists. Consider such a decomposition and denote by $N$ and $N_i$ the simply connected Lie groups with Lie algebras $\mn$ and $\mn_i$ respectively. Then $(N,g)$ is isometric to the Riemannian product 
$$(N_1,g_1)\times\ldots\times (N_k,g_k)\times \mathbb{R}^{\dim(\ma)},$$
and moreover $(N_i,g_i)$ is an irreducible Riemannian manifold for each $i\in\{1,\ldots,k\}$. Indeed, if some $(N_i,g_i)$ were reducible, then the restriction of the metric $g_i$ to the factors would define parallel symmetric tensors on $(N_i,g_i)$, which by Theorem \ref{teo.parallel} have to be left-invariant. By Proposition \ref{pro.paralleltensors} this would give a decomposition of $\mn_i$ as an orthogonal direct sum of ideals, thus contradicting its irreducibility.

Therefore, the uniqueness statement follows from the uniqueness (up to reordering) of the de Rham decomposition of Riemannian manifolds.
\end{proof}

\bigskip
\bigskip

\bibliographystyle{plain}
\bibliography{biblio}

\begin{thebibliography}{10}

\bibitem{AD}
A.~{Andrada} and I.G. {Dotti}.
\newblock {Conformal Killing-Yano 2-forms.}
\newblock {\em {Differ. Geom. Appl.}}, 58:103--119, 2018.

\bibitem{BDS}
M.L. {Barberis}, I.G. {Dotti}, and O.~{Santill\'an}.
\newblock {The Killing-Yano equation on Lie groups.}
\newblock {\em {Classical Quantum Gravity}}, 29(6):1--10, 2012.

\bibitem{BMMT}
A.~Bolsinov, V.~Matveev, E.~Miranda, and S.~Tabachnikov.
\newblock Open problems, questions, and challenges in finite-dimensional
  integrable systems.
\newblock {\em Philos. Trans. Roy. Soc. A}, 376(2131), 2018.

\bibitem{Bu03}
L.~{Butler}.
\newblock {Integrable geodesic flows with wild first integrals: the case of
  two-step nilmanifolds.}
\newblock {\em {Ergodic Theory Dyn. Syst.}}, 23(3):771--797, 2003.

\bibitem{EB}
P.~Eberlein.
\newblock Geometry of {$2$}-step nilpotent groups with a left invariant metric.
\newblock {\em Ann. Sci. \'Ecole Norm. Sup. (4)}, 27(5):611--660, 1994.

\bibitem{He18}
K.~Heil.
\newblock {\em Killing and conformal Killing tensors}.
\newblock {Dissertation an der Universit\"at Stuttgart}, 2018.

\bibitem{HMS16}
K.~Heil, A.~Moroianu, and U.~Semmelmann.
\newblock {Killing and conformal Killing tensors}.
\newblock {\em J. Geom. Phys.}, 106:383--400, 2016.

\bibitem{HMS17}
K.~Heil, A.~Moroianu, and U.~Semmelmann.
\newblock {Killing tensors on tori.}
\newblock {\em {J. Geom. Phys.}}, 117:1--6, 2017.

\bibitem{HE}
S.~Helgason.
\newblock {\em Differential geometry, {L}ie groups, and symmetric spaces},
  volume~80 of {\em Pure and Applied Mathematics}.
\newblock Academic Press Inc. [Harcourt Brace Jovanovich Publishers], New York,
  1978.

\bibitem{KOR}
A.~{Kocsard}, G.P. {Ovando}, and S.~{Reggiani}.
\newblock {On first integrals of the geodesic flow on Heisenberg nilmanifolds.}
\newblock {\em {Differ. Geom. Appl.}}, 49:496--509, 2016.

\bibitem{MS}
V.~Matveev and V.~Shevchishin.
\newblock Two-dimensional superintegrable metrics with one linear and one cubic
  integral.
\newblock {\em J. Geom. Phys.}, 61(8):1353--1377, 2011.

\bibitem{MF}
A.S. Mischenko and A.T. Fomenko.
\newblock Euler equations on finite-dimensional lie groups.
\newblock {\em Math. USSR-Izv.}, 12(2):371--389, 1978.

\bibitem{MSS}
R.~{Montgomery}, M.~{Shapiro}, and A.~{Stolin}.
\newblock {A nonintegrable sub-Riemannian geodesic flow on a Carnot group.}
\newblock {\em {J. Dyn. Control Syst.}}, 3(4):519--530, 1997.

\bibitem{Ov18}
G.~Ovando.
\newblock Integrable geodesic flows on nilmanifolds: new examples.
\newblock {\em arXiv 1708.09457}.

\bibitem{RA}
M.S. Raghunathan.
\newblock {\em Discrete subgroups of {L}ie Groups}.
\newblock Springer, 1972.

\bibitem{Se03}
U.~Semmelmann.
\newblock {Conformal Killing forms on Riemannian manifolds}.
\newblock {\em {Math. Z.}}, 245:503--527, 2013.

\bibitem{Ta83}
M.~{Takeuchi}.
\newblock {Killing tensor fields on spaces of constant curvature.}
\newblock {\em {Tsukuba J. Math.}}, 7:233--255, 1983.

\bibitem{Th86}
G.~Thompson.
\newblock Killing tensors in spaces of constant curvature.
\newblock {\em J. Math. Phys.}, 27(11):2693--2699, 1986.

\bibitem{VAR}
V.S. Varadarajan.
\newblock {\em Lie groups, {L}ie algebras, and their representations}, volume
  102 of {\em Graduate Texts in Mathematics}.
\newblock Springer-Verlag, New York, 1984.
\newblock Reprint of the 1974 edition.

\bibitem{Wo63}
J.~{Wolf}.
\newblock {On locally symmetric spaces of non-negative curvature and certain
  other locally homogeneous spaces.}
\newblock {\em {Comment. Math. Helv.}}, 37:266--295, 1963.

\end{thebibliography}

\end{document}